%% file: Algebraic_Genus.tex
\documentclass[10pt]{amsart}
\usepackage[margin=1in]{geometry}
\usepackage{graphicx}
\usepackage[dvipsnames]{xcolor}
\usepackage{amssymb}
\usepackage{enumitem}
\usepackage{array,multirow}
\usepackage[all]{xy}
\usepackage{amsmath}
\usepackage{amsthm}
\usepackage{hyperref}
\usepackage[nameinlink]{cleveref}

\theoremstyle{plain}

\newtheorem{thm}{Theorem}[section]
\newtheorem{cor}[thm]{Corollary}
\newtheorem{prop}[thm]{Proposition}
\newtheorem{lemma}[thm]{Lemma}
\newtheorem{example}[thm]{Example}
\newtheorem{definition}[thm]{Definition}
\newtheorem{conj}[thm]{Conjecture}

\newtheorem{prob}[thm]{Problem}
\newtheorem{remark}[thm]{Remark}

\hypersetup{colorlinks=true,linkcolor=Cyan,citecolor=Violet}

\newenvironment{Example}{\begin{example}\rm}{\end{example}}

\newcommand{\C}{{\ensuremath{\mathbb{C}}}}

\newcommand{\N}{{\ensuremath{\mathbb{N}}}}
\newcommand{\Z}{{\ensuremath{\mathbb{Z}}}}
\newcommand{\Q}{{\ensuremath{\mathbb{Q}}}}
\newcommand{\galg}{{\ensuremath{g_{\mathrm{alg}}}}}
\newcommand{\g}{{\ensuremath{g_4^\mathrm{top}}}}
\newcommand{\gZ}{{\ensuremath{g_\Z}}}
\newcommand{\gsm}{{\ensuremath{g_4^{\mathrm{sm}}}}}
\DeclareMathOperator{\sig}{\sigma}
\DeclareMathOperator{\Sig}{\Sigma}
\DeclareMathOperator{\om}{\omega}

\newcolumntype{C}{>{\boldmath\textbf}c}

\def\et{\;\mbox{and}\;}

\title[A note on the topological slice genus of satellite knots]{A note on the topological slice genus of satellite knots}
\author{Peter Feller}
\author{Allison N. Miller}
\author{Juanita Pinzon-Caicedo}

\begin{document}

\begin{abstract}
This paper presents evidence supporting the surprising conjecture that in the topological category the slice genus of a satellite knot $P(K)$ is bounded above by the sum of the slice genera of $K$ and $P(U)$. Our main result establishes this conjecture for a variant of the topological slice genus, the $\mathbb{Z}$-slice genus.
As an application, we show that the $(n,1)$-cable of any 3-genus 1 knot (e.g.~the figure 8 or trefoil knot) has topological slice genus at most 1. Further, we show that the lower bounds on the slice genus coming from the Tristram-Levine and Casson-Gordon signatures cannot be used to disprove the conjecture.
Notably, the conjectured upper bound does not involve the algebraic winding number of the pattern $P$. This stands in stark contrast with the smooth category, where for example there are many genus 1 knots whose $(n,1)$-cables have arbitrarily large smooth 4-genera.  \end{abstract}
\maketitle

\section{Introduction}

The behavior of the Seifert genera of knots under the satellite construction is completely understood. Let $P$ be a pattern, i.e.~a knot in a solid torus,  with (algebraic) winding number  $w$, let $K$ be a knot in $S^3$, and let $P(K)$ denote the resulting satellite knot in $S^3$; see \Cref{fig:4-1-cableoftrefoil} for an example and see \Cref{sec:Def} for precise definitions. A result of Schubert~\cite{schubert} states that for any pattern $P$  with winding number  $w$, there exists a constant $g_3(P)$---a version of the 3-genus for patterns---such that for any nontrivial knot $K$ in $S^3$ we have
\[g_3(P(K))= g_3(P)+ |w| g_3(K).\]

\begin{figure}[h]
\def\svgwidth{0.125\textwidth}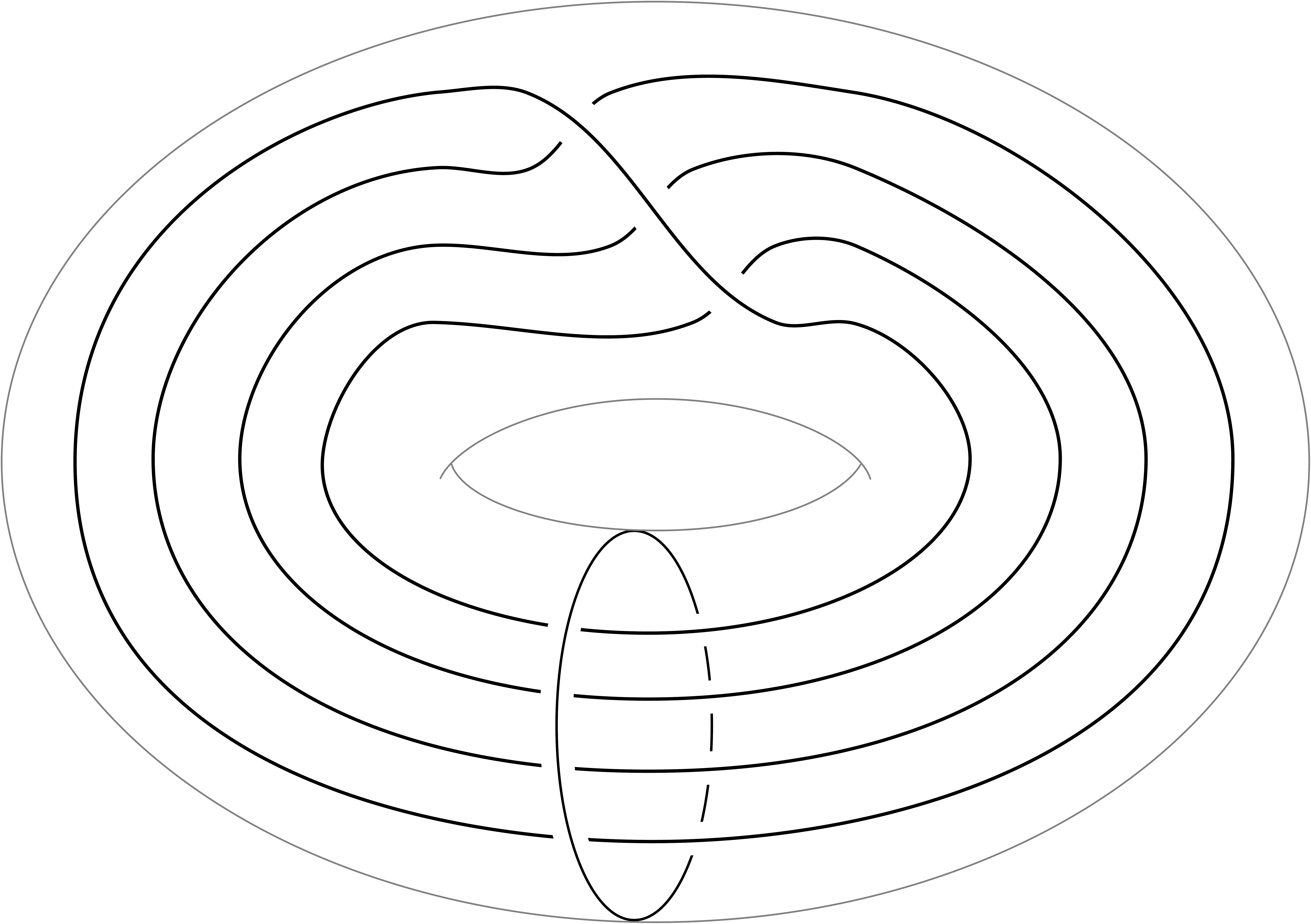\hspace{1cm}
\def\svgwidth{0.125\textwidth}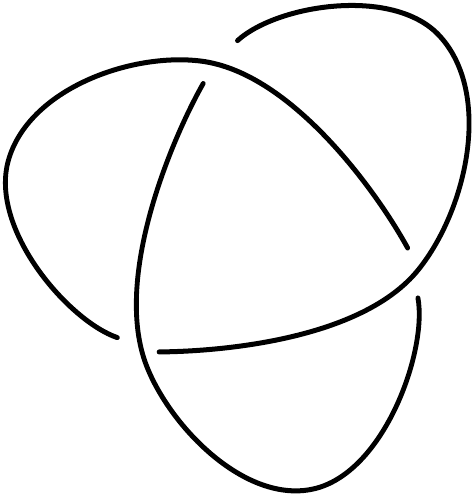\hspace{1cm}
\def\svgwidth{0.125\textwidth}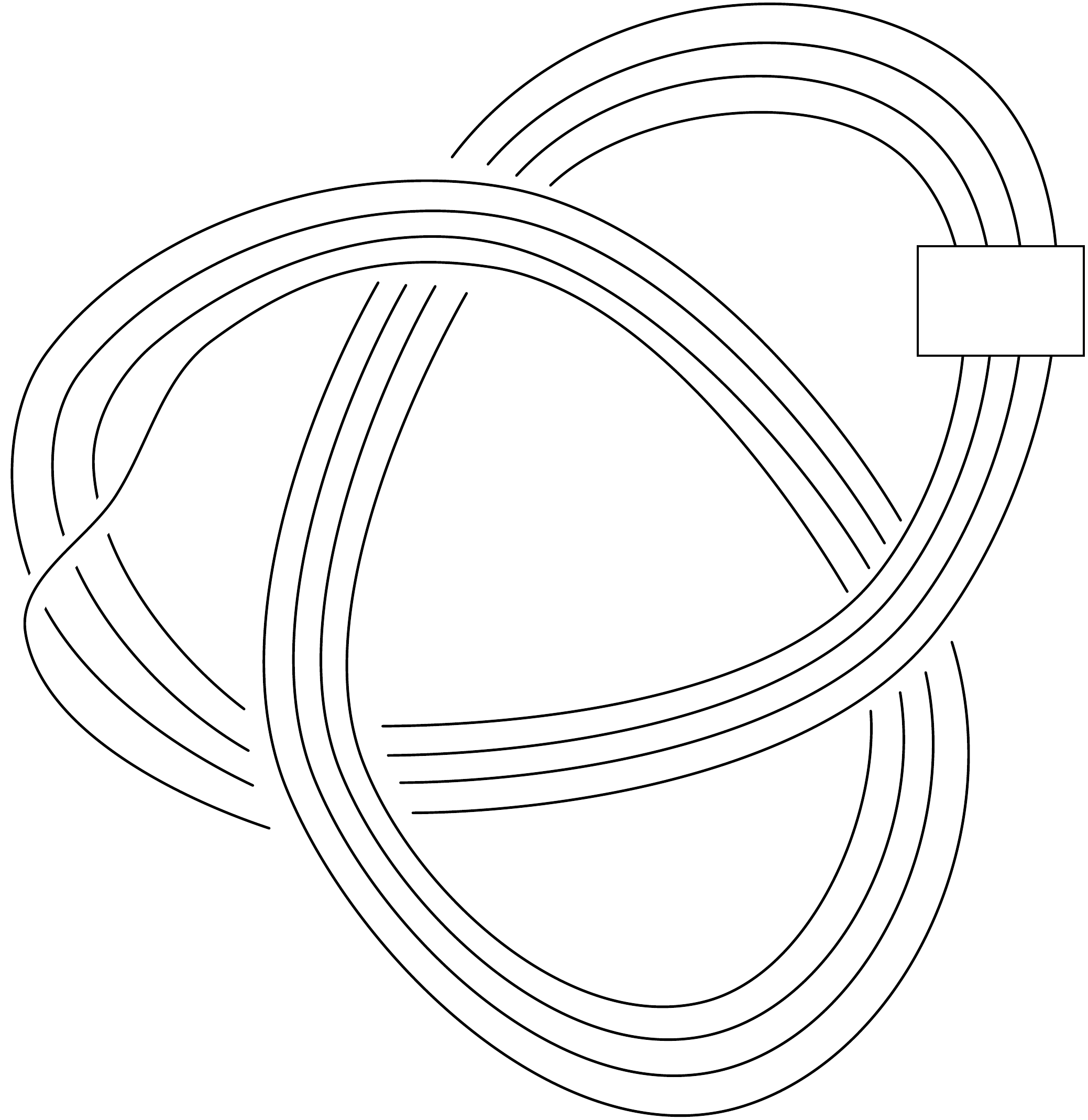
\caption{A pattern $P=C_{4,1}$ with $|w|=4$  (left), a knot $K$ (center), and the  satellite $P(K)$ (right). The box on the right indicates three negative full twists.}\label{fig:4-1-cableoftrefoil}
\end{figure}

Unsurprisingly, the 4-dimensional situation is more complicated. We remind the reader that the topological 4-genus of $K$, denoted $\g(K)$, is the minimal genus of any locally flatly embedded orientable surface in $B^4$ with boundary $K$, and the smooth 4-genus $\gsm(K)$ is analogously defined. It is not hard to show that a bound
\[g_4(P(K)) \leq g_4(P)+ |w|g_4(K)\]
 holds in both categories, where we emphasize that $g_4(P)$ is a version of the 4-genus for the pattern $P$ and is generally strictly larger than $g_4(P(U))$. In the smooth category, the  naive expectation that $\gsm(P(K))$ is approximately $|w| \gsm(K)$ often holds: for any winding number $w$ pattern $P$, we have
\begin{align}\label{eqn:smoothlimit}\displaystyle \lim_{n \to \infty} \frac{ \gsm(P(T_{2,2n+1}))}{\gsm(T_{2,2n+1})}=|w|.\end{align}
Moreover, for any $w, m \in \mathbb{N}$ there exists a winding number $w$ pattern $Q=Q_{w,m}$ and infinitely many knots $J$ such that
\begin{align} \label{eqn:specialpattern}\gsm(Q(J)) =\left(\gsm(Q(U))+ |w| \gsm(J)\right)+m.\end{align}
We expect that these observations are known to the experts, but for completeness we prove them in \Cref{sec:Exs}.\\

The satellite operation seems to affect $\g$  very differently.  In this paper we give evidence for the  surprising idea that the winding number of $P$ essentially does not contribute to $\g(P(K))$.

\begin{conj}\label{conj:(1)holds}
For any pattern $P$ and knot $K$,
$
\g(P(K)) \leq \g(P(U))+ \g(K).
$
\end{conj}

The simplest lower bounds on the topological 4-genus of a knot come from Tristam-Levine signatures, and we see in \Cref{sec:CG} that the satellite formula of Litherland \cite{Litherland_satellite} quickly implies that the lower bound for $\g$ given by the Tristram-Levine signature  cannot be used to establish that a pair $P$ and $K$ fails to satisfy \Cref{conj:(1)holds}.  As further evidence towards \Cref{conj:(1)holds}, in \Cref{sec:CG} we also consider Gilmer's lower bound for $\g$  \cite{Gilmer_slicegenus} in terms of Casson-Gordon signature invariants \cite{cg-slice,cg-cob}, and in \Cref{prop:nocg} we show that this bound cannot be used to show that a pair $(P,K)$ fails to satisfy \Cref{conj:(1)holds}.

\subsection*{Main result: an inequality for the $\Z$-slice genus of satellites}
We show that the inequality of \Cref{conj:(1)holds} holds for the \emph{topological $\Z$-slice genus} $\gZ$, an analog of $\g$. The $\Z$-slice genus 
is defined as
\[\gZ(K):= \min\left\{\mathrm{genus}(F)\, : \,
F\hookrightarrow B^4 \text{ is an oriented locally-flat surface with }\partial F= K\text{ and } \pi_1(B^4 \setminus F) \cong \Z \right\}.\]
Observe that $\g\leq  \gZ$ by definition and that $\gZ\leq g_3$ since the complement of a Seifert surface that was properly pushed into the 4-ball has fundamental group $\Z$ (see \cite[Proposition 6.2.1]{GompfStipsicz} or \cite[Proof of Theorem 1]{FellerLewark_16} for more details). Notice also that $\gZ(K)=0$ if and only if $\Delta_K(t)=1$~\cite[Theorem~1.13]{Freedman_82_TheTopOfFour-dimensionalManifolds}.\\

Our main theorem reads as follows.
\begin{thm}\label{thm:mainIntro}
  For any pattern $P$ and  knot $K$,
  $\gZ(P(K))\leq \gZ(P(U))+\gZ(K)$.
\end{thm}

In fact, our proof of \Cref{thm:mainIntro} implies that when $w(P)=0$ we obtain $\gZ(P(K)) = \gZ(P(U))$ and when $w(P)= \pm1$ we have $\gZ(P(K))= \gZ(P(U) \#K)$. This second fact is interesting given that an unresolved problem asks whether $P(K)$ and $P(U)\#K$ must be  topologically concordant when $w(P)=+1$. However, we think that the result is most surprising for $|w(P)|>1$, where it stands in contrast with smooth results such as~\eqref{eqn:smoothlimit} and~\eqref{eqn:specialpattern}.
\\

Notice that \Cref{thm:mainIntro} immediately gives upper bounds for the topological 4-ball genus of a satellite knot. For example,  we have the following unexpected result.

\begin{Example}[The $(n,1)$-cable of the trefoil.]\label{Ex:trefoilandfig8}
For a knot $K$ and $n>0$, let $C_{n,1}(K)$ denote the $(n,1)$-cable of $K$ and observe that \Cref{thm:mainIntro} implies that
\[\g(C_{n,1}(K))\leq \gZ(C_{n,1}(U))+ \gZ(K)\leq g_3(K).\]
A simple Tristram-Levine signature computation at an appropriate $\omega \in S^1$ (see the proof of \Cref{cor:simpleg4equality}) shows that
$C_{n,1}(T_{2,3})$ is not slice and so $\g(C_{n,1}(T_{2,3}))=1$ for all $n>0$.
This is particularly surprising given that $\g(T_{2,3})= \gsm(T_{2,3})=1$ and $\gsm(C_{n,1}(T_{2,3}))=n$.
\end{Example}

We also obtain the following explicit difference with~\eqref{eqn:smoothlimit}, which we prove in \Cref{sec:Exs}.
\begin{cor}\label{cor:limitis1intop}
Let $P$ be a pattern of winding number $w$. Then
$$ \displaystyle \lim_{n \to \infty} \frac{ \g(P(T_{2,2n+1}))}{\g(T_{2,2n+1})}=\left\{\begin{array}{cc}1, & w\neq 0 \\ 0, & w=0 \end{array} \right..$$
\end{cor}

As another explicit example of the difference between the smooth and topological categories, in \Cref{Ex:2cablesof2strandedtorusknots}, we see that iterative 2-cabling of $T_{2,p}$ torus knots yields families of knots $K_n$ that are closures of positive braids for which \Cref{thm:mainIntro} immediately shows $\lim_{n\to\infty}\frac{\g(K_n)}{\gsm(K_n)}\leq\frac{2}{3}$.
Previous work on the ratio between the smooth and topological genera of positive braid closures has relied on explicit example-based calculations, see~\cite{Rudolph_84_SomeTopLocFlatSurf,BaaderFellerLewarkLiechti_15},
but our arguments allow us to improve previous bounds without  computing specific Seifert matrices.
\\

We remark that the optimal upper bound for $\g$ coming from \Cref{thm:mainIntro} is
\[ \g(P(K)) \leq \gZ(P(U))+ \min\{|w|, 1\} \gZ^c(K),\]
where $\gZ^c(K)$ is the concordance $\Z$-slice genus of $K$, that is, the minimal $\gZ(J)$ over all knots $J$ topologically concordant to $K$. This follows immediately from the observation that if $K$ and $J$ are topologically concordant, then $P(K)$ and $P(J)$ are also topologically concordant and  so $\g(P(K))$ equals $\g(P(J))$.
\\

Many particularly nice examples, including \Cref{Ex:trefoilandfig8}, fall into the following setting.
\begin{cor}\label{cor:unknottedpattern}
For every knot $K$ and pattern $P$ with $\Delta_{P(U)}(t)=1$, we have $\g(P(K))\leq g_3(K)$.
\end{cor}

Litherland's formula for the Tristram-Levine signature of a satellite knot allows us to construct many examples where the inequality of \Cref{cor:unknottedpattern} becomes equality.

\begin{cor}\label{cor:simpleg4equality}
Let $P$ be a pattern of nonzero winding number with $\Delta_{P(U)}(t)=1$.
Let $K$ be any knot such that ${g_3}(K)= |2\sigma_{\omega}(K)|$ for some  $\omega \in S^1$ with
$\Delta_K(\omega) \neq 0$. Then
$\g(P(K))=g_3(K)=\g(K)$.
\end{cor}
We remark that the hypothesis on the winding number of $P$ is necessary: if $P$ has winding number $0$ and $\Delta_{P(U)}(t)=1$, then  $P(K)$ has trivial Alexander polynomial and so $\g(P(K))=0$ for any knot $K$.
\begin{proof}
Our assumption on $\sigma_{\omega}(K)$ implies that $\gZ(K)= g_3(K)$, and
\Cref{cor:unknottedpattern} further implies that
\[\g(P(K))\leq \gZ(P(U))+\gZ(K)= \gZ(U)+\gZ(K)=0+g_3(K).
\]
Now let $\xi \in S^1$ be a prime power root of unity such that no root of $\Delta_{P(K)}(t)$ lies between $\xi^n$ and $\omega$, where $n=|w|$ is the absolute value of the winding number of $P$. Observe that
\[2\g(P(K)) \geq |\sigma_{\xi}(P(K))|= |\sigma_{\xi}(P(U))+ \sigma_{\xi^n}(K)| = |0+ \sigma_{\omega}(K)|= 2\g(K). \qedhere
\]
\end{proof}

Unsurprisingly, one can find many examples where the bounds on $\g(P(K))$ coming from \Cref{thm:mainIntro} are far from sharp. For instance if $P$ is a pattern with geometric winding number 1 and such that $\g(P(U))=n$, then
\[0= \g(P(U) \# -P(U))= \g(P(-P(U))< \g(P(U))+ \g(-P(U))=2n.
\]
There are also many examples of pairs $(P,K)$ where the topological 4-genus of $P(K)$ cannot be determined by combining the upper bounds coming from \Cref{thm:mainIntro} with the known lower bounds. We give a particularly interesting family that may relate to \Cref{conj:(1)holds}.

\begin{figure}[h]
\includegraphics[height=3cm]{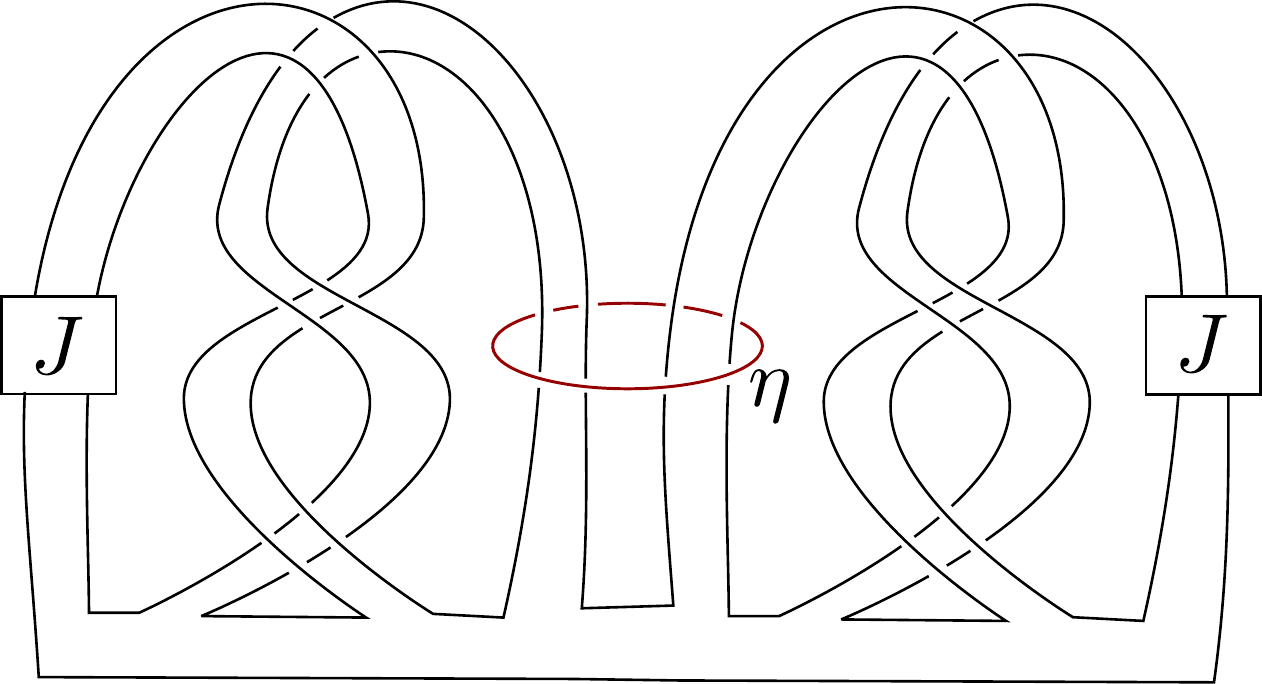}
\caption{\label{fig:PJ}The pattern $P_J$, which depends on the choice of an auxiliary knot $J$ and has algebraic winding number equal to $0$.}
\end{figure}

\begin{Example}\label{exl:pj}
Let $P_J$ be the pattern shown in \Cref{fig:PJ}, described as a knot in the complement of the unknot $\eta$.
Observe that since $P_J(U)$ has $H_1(\Sigma_2(P_J(U))) \cong (\Z_3)^4$, we have that
\[2= \frac{1}{2}(4) \leq \gZ(P_J(U)) \leq g_3(P_J(U))= 2,\]
where for the first inequality we used that half the minimal number of generators for the first homology of the double branched cover of a knot is a lower bound for $\gZ$, see \cite[Proposition~12.ii)]{FellerLewark_16} and \cite[Corollary~1.5]{FellerLewark_19}.
So the best algebraic bound we can obtain is $\g(P_J(K)) \leq 2$, which also follows immediately from considering the genus 2 Seifert surface for $P_J(U)$ in the complement of $\eta$. \Cref{conj:(1)holds} suggests that in fact $$\g(P_J(K)) \leq \min\{2,\g(K)\}.$$
We remark that while for many choices of $J$ (e.g.~$J=\#^n T_{2,3}$ for large $n$) one can use Casson-Gordon signatures to prove that $P_J(T_{2,3})$ is not slice, \Cref{prop:nocg} shows that it is not possible to use Gilmer's version of the Casson-Gordon signature bounds to establish that $\g(P_J(T_{2,3}))>1= \g(P_J(U))+\g(T_{2,3})$.
\end{Example}

Given the gap between the known lower and upper bounds on $\g$ in the case of $P_J(T_{2,3}))$, we propose the following as a stimulus for future work.
\begin{prob}
For some non-slice knot $J$, determine $\g(P_J(T_{2,3})) \in \{1,2\}$.
\end{prob}

\subsection*{Outline of Proof of \Cref{thm:mainIntro}}
Our proof of \Cref{thm:mainIntro} will be Seifert matrix based: we establish the inequality of \Cref{thm:mainIntro} by proving in \Cref{prop:maininequality} that the corresponding inequality for the algebraic genus $\galg$ holds. This latter quantity is defined in terms of S-equivalence classes of links by Feller-Lewark in \cite{FellerLewark_16} and was shown to be equal to the topological $\Z$-slice genus in ~\cite{FellerLewark_19}. While we only considered connected patterns thus far, we note that one could also consider patterns $P$ with multiple components and thus satellites $P(K)$ that are links of multiple components. Our proof of \Cref{prop:maininequality} holds equally well for multi-component patterns. However, we warn the reader that in the setting of multiple components it is known that $\gZ\leq\galg$ but not whether $\gZ=\galg$. As a result, we do not know that \Cref{thm:mainIntro} holds for multi-component patterns.

Duncan McCoy~\cite{McCoy_19} has an alternative proof of \Cref{thm:mainIntro} which relies on his recent work analysing the behavior of $\galg$ under so-called `null homologous twisting operations'. Additionally, in the final stages of the preparation of this manuscript we found yet a third way to prove \Cref{thm:mainIntro} by combining a result of Livingston and Melvin~\cite{LivingstonMelvin85} about the Blanchfield pairing and the recent characterization of $\gZ$ in terms of the Blanchfield pairing given in~\cite[Theorem~1.1]{FellerLewark_19}; see our Blanchfield pairing perspective below.
\\

We end the introduction with remarks on different perspectives on $\g$ and the satellite operation.

\subsection*{The Alexander polynomial perspective.}
Besides the Tristram-Levine signature $\sigma_\omega(K)$ (compare with~\Cref{sec:CG}),
the Alexander polynomial $\Delta_K(t)$  is another classical knot invariant that has a simple behavior with respect to satellite operations and provides bounds (upper rather than lower) for $\g$. Namely, a formula of Litherland \cite{Litherland_satellite} states that for all patterns $P$ and knots $K$
\begin{align}\label{eq:alexsatellite}
\Delta_{P(K)}(t)=\Delta_{P(U)}(t)\Delta_K(t^{|w|}).
\end{align}

In addition, as a consequence of Freedman's Disc Embedding theorem (see~\cite[Theorem~1.13]{Freedman_82_TheTopOfFour-dimensionalManifolds} and~\cite[Theorem~1]{Feller_15_DegAlexUpperBoundTopSliceGenus}), we have $$2\gZ(K)\leq \deg(\Delta_K(t)).$$

Considering the addition formula for the degree of the Alexander polynomial coming from \Cref{eq:alexsatellite} and the relationship of the degree to $\g$ and $\gZ$, it is natural to wonder if it is true that \[\g(P(K))\leq \g(P(U))+|w|\g(K).\]

However, there certainly exist winding number 0 patterns $P$ with $P(U)$  slice such that $P(K)$ is not slice for appropriate choices of $K$, see e.g.~\Cref{exl:pj} above.
Moreover, when $w \neq 0$ this  inequality is subsumed by \Cref{conj:(1)holds}.

\subsection*{The Blanchfield pairing perspective}
Recall that for a knot $K\subset S^3$ the Alexander module $A_K$ is the first integer homology of the infinite cyclic cover of the knot complement viewed as a $\Z[t^{\pm1}]$-module via the deck group action. The Blanchfield pairing $Bl(K)$ of $K$ is a a nonsingular, hermitian, sesquilinear form
\[Bl(K)\colon A_K\times A_K\mapsto \Q(t)/\Z[t^{\pm1}],\]
that is linear in the first variable, and antilinear in the second variable with respect to the involution induced by $t\mapsto t^{-1}$. The Blanchfield form $Bl(K)$ can be expressed entirely in terms of a Seifert matrix for $K$. Moreover, two Seifert matrices are S-equivalent if and only if they determine isomorphic Blanchfield forms. Here, isomorphic means that for two knots $K,K'$ there exists a $\Z[t^{\pm1}]$-module isomorphism $\phi:A_K\to A_{K'}$ such that $Bl(K')(\phi(x),\phi(y))=Bl(K)(x,y)$ for all $x,y\in A_K$. See \cite{kearton,Ko89,FriedlPowell} for more details.\\

In~\cite[Theorem~2]{LivingstonMelvin85} Livingston and Melvin show that \[Bl(P(K))(t)\cong Bl(P(U))(t)\otimes Bl(K)(t^{w}), \]
generalizing a result of Litherland \cite{Litherland_satellite}, where this was established for $\Q[t^{\pm1}]$ coefficients.\\

This allows to provide another proof of \Cref{thm:mainIntro}. Namely, the recent characterization of $\gZ=\galg$ in terms of the Blanchfield pairing from~\cite[Theorem~1.1]{FellerLewark_19} implies that the inequality $\gZ(P(K))\leq \gZ(P(U))+\gZ(K)$ follows from $Bl(P(K))(t)\cong Bl(P(U))(t)\otimes Bl(K)(t^w)$. We do not provide details of this here as we have an elementary matrix based proof that works more generally for satellites of multiple components, 
and that does not rely on the heavy-duty inputs that are crucial for the characterization of $\gZ$ from~\cite{FellerLewark_19}, such as the Disc Embedding Theorem.

\subsection*{Acknowledgements}
We are very grateful for the exchange with Duncan McCoy, who immediately after learning of \Cref{thm:mainIntro} observed an alternative proof for it based on the behavior of the algebraic genus under null homologous twisting operations; compare~\cite{McCoy_19}.

\section{Definitions and main result for the algebraic genus}\label{sec:Def}

In this section we establish an inequality relating the $\Z$-genera of $P(U)$, $K$, and $P(K)$. We do so by establishing an inequality between their algebraic genera, defined below. Since $\gZ$ and $\galg$ are the same for knots, this will translate back to an inequality for $\gZ$ when $P$ is a one-component pattern as in the case of interest. The advantage of working with $\galg$ is that one can work with algebraic manipulations of Seifert matrices, which can be taken to have a particular form for satellites.\\

We start by recalling the relevant definitions and properties.

\begin{definition}
For a link $L\subseteq S^3$ with $r$ components, one defines its \emph{algebraic genus} as
\[\galg(L)=\min\left\{\;\;\frac{m-2n-r+1}{2}\;\;\; \middle | \;\;\;\parbox{0.6\linewidth}{There exists a Seifert surface $F$ for $L$ with $m\times m$ Seifert matrix of the form $\left[\begin{array}{cc}B & * \\ * & *\end{array}\right]$, where $B$ is a $2n\times 2n$ matrix satisfying $\det(tB-B^T)=t^n$.}\;\;\right\}\]

A Seifert surface $F$ for $L$ is said to \emph{realize} the algebraic genus $\galg(L)$ if it has a Seifert matrix as above such that $\frac{m-2n-r+1}{2}=\galg(L)$.
\end{definition}

The definition is chosen such that a knot $K$ has $\galg(K)=0$ if and only if it has trivial Alexander polynomial. Indeed, $2n\times 2n$ matrices $B$ with $\det(tB-B^T)=t^n$ for some $n \in \mathbb{Z}$ are exactly the matrices that occur as Seifert matrices of knots with trivial Alexander polynomial. We call such a $B$ an \emph{Alexander trivial matrix} or, if it is a diagonal sub-block of a larger matrix, an \emph{Alexander trivial submatrix}. A key feature of the algebraic genus is that $\gZ(L)\leq\galg(L)$ for all links $L$, so $\galg$  provides a Seifert matrix based upper bound on $\gZ$ and thus $\g$; see~\cite{FellerLewark_16}. Furthermore, $\gZ(K)=\galg(K)$ for all knots $K$~\cite[Corollary~1.5]{FellerLewark_19}, which is what we use to translate statements about $\galg$ to ones about $\gZ$.

\begin{definition}\label{def:sat} Let $P \sqcup \eta\subset S^3$ be a link of $r+1\geq 2$ components with $\eta$ an unknot such that $P$ is contained in the interior of the solid torus $V=S^3\setminus N(\eta)$. Denote by $c$ a simple closed curve representing the generator of $H_1(V;\Z)$ specified by the condition $lk(c,\eta)=+1$. Let $K\subset S^3$ be a knot and let $h\colon V \to N(K)\subset S^3$ be an orientation preserving homeomorphism  taking $c$ to $K$ and a 0-framed longitude of $c$ to a 0-framed longitude of $K$. The image of $P$ under $h$, denoted by $P(K)$, is the \emph{satellite link} with \emph{pattern} $P$ and \emph{companion} $K$. The \emph{algebraic winding number} or \emph{winding number} of $P$ is defined as $w=lk(P,\eta)$.
\end{definition}
The reader may be used to requiring $P$ to be a connected pattern, i.e.~restricting to $r=1$. In this section, we consider general patterns with $r\geq1$, which in general have that $P(K)$ is a link rather than a knot. However, in all other sections we only consider classical patterns with $r=1$.

\begin{remark} Note that without loss of generality, it is enough to consider patterns with nonnegative winding number. Indeed, if $P$ is a pattern with negative winding number then $w=lk(P,\eta)<0$, and so $lk(P^{rev},\eta)=-w>0$ and $P^{rev}$ has positive winding number. Furthermore, since $P^{rev}(K)= P(K)^{rev}$, any notion of genus agrees on $P(K)$ and $P^{rev}(K)$.
\end{remark}

\subsection*{Main result about the algebraic genus of satellites}
Our main theorem about the algebraic genus of satellites is the following.

\begin{prop}\label{prop:maininequality} For a satellite link $P(K)$ with pattern $P$ and companion $K$, the following inequality holds
\[\galg(P(K))\leq \galg(P)+\min\{|w|,1\}\galg(K).\]

In fact, for $|w|=1$ and $w=0$, we have that $P(K)$ is S-equivalent to $P(U)\sharp K$ and $P(U)$, respectively.
\end{prop}

Before we provide the proof of \Cref{prop:maininequality}, we derive \Cref{thm:mainIntro} from it.
\begin{proof}[Proof of \Cref{thm:mainIntro}]
Let $P$ be a one component pattern and $K$ be a knot.
Then $\gZ=\galg$ for $P(K)$, $P(U)$ and $K$, since they are all knots \cite[Corollary~1.5]{FellerLewark_19}. Using these equalities, \Cref{thm:mainIntro} follows immediately from \Cref{prop:maininequality}.
\end{proof}

Our proof of \Cref{prop:maininequality} uses a construction of a Seifert surface for $P(K)$ similar to the one in~\cite[Chapter~6, Theorem~6.15]{Lickorish_97}, and illustrated below, with some additional attention paid to realizing $\galg$.
\begin{figure}[h!]
\includegraphics[height=2.5cm]
{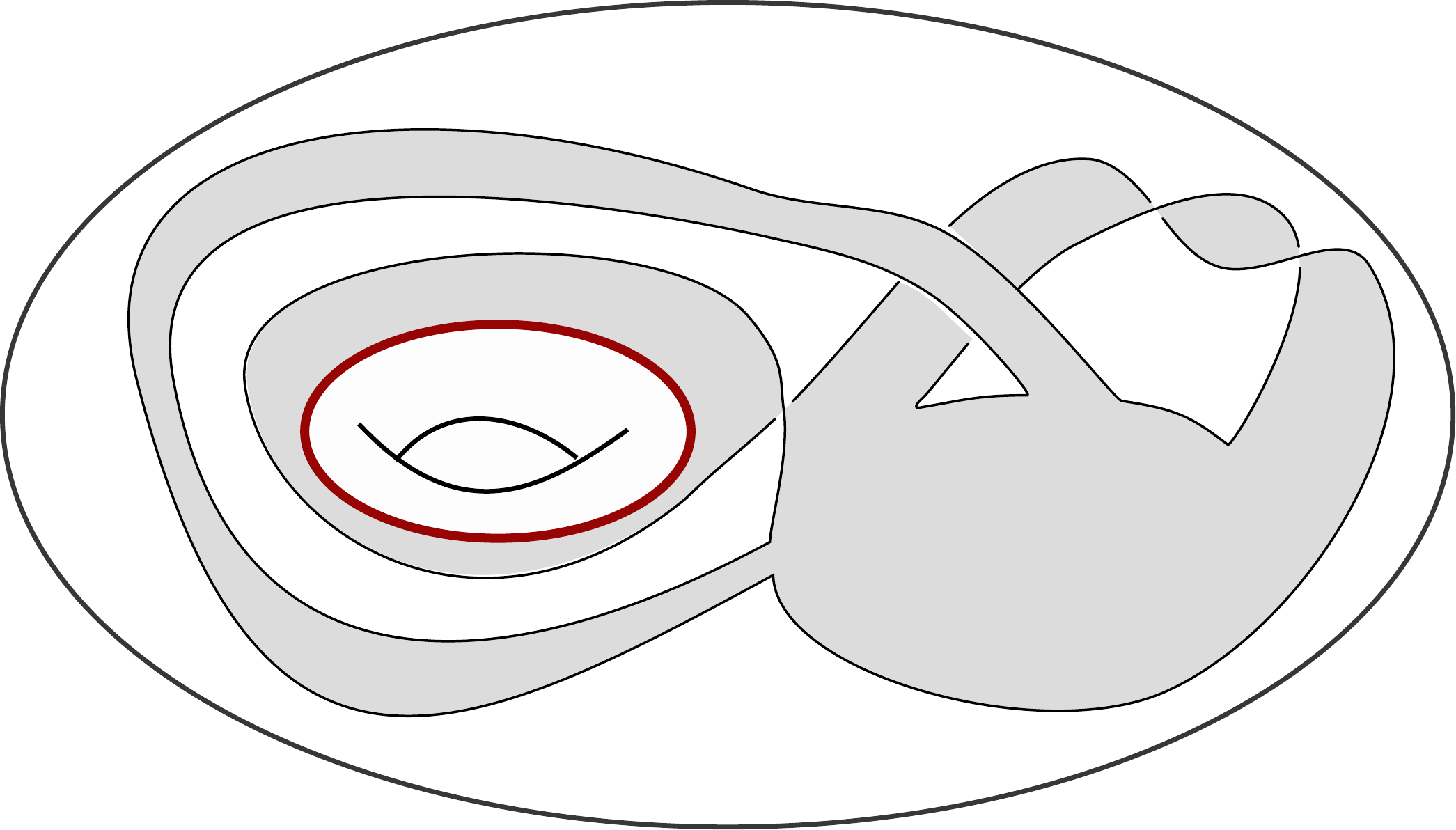}
\qquad
\includegraphics[height=2.5cm]
{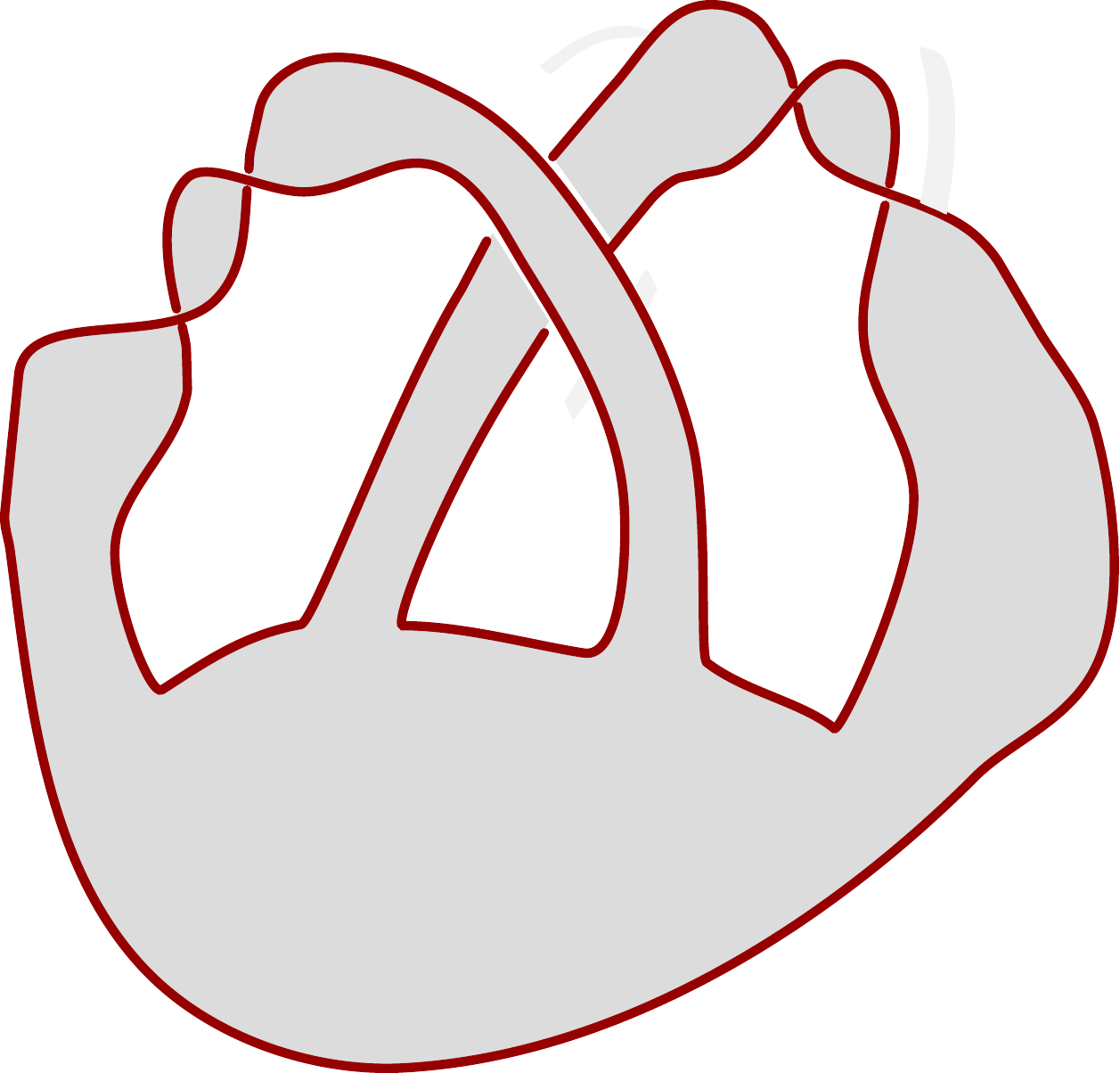}
\qquad
\includegraphics[height=3.25cm]
{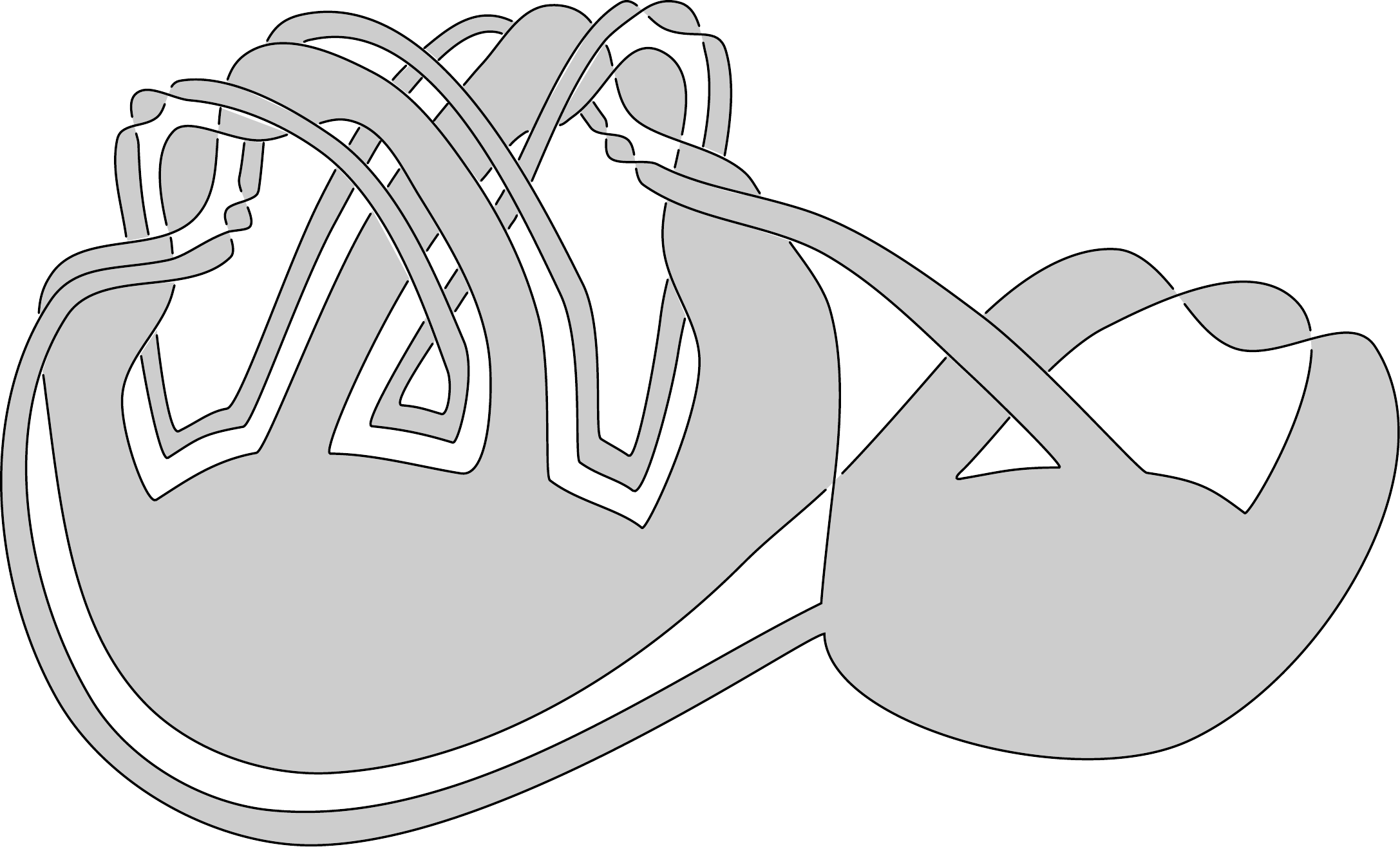}
\caption{A Seifert surface for a pattern $P$ (left) and a Seifert surface for a knot $K$ (center) combine to give a Seifert surface for $P(K)$ (right).}\label{fig:seifertexample}
\end{figure}

\begin{lemma}\label{lemma:NiceSSforP} Let $P\, \sqcup\, \eta$ be a pattern with winding number $w\geq 0$, and let $l$ denote a chosen 0-framed longitude in the boundary of $V=S^3\setminus N(\eta)$.
 There exists a Seifert surface $G\subset S^3\setminus N(\eta)$ for the link $P\sqcup wl$ such that $G\cup_{wl} wD^2$ is a Seifert surface for $P(U)$ that realizes $g_{alg}(P(U))$. Here $wl$ and $wD^2$ denote respectively $w$ parallel copies of $l$ and $D^2$.
\end{lemma}
\begin{proof} The link $P(U)$ is obtained by regarding the pattern $P$ as a link in $S^3$, forgetting about the effect of the unknotted component $\eta$. Let $F$ be a Seifert surface for $P(U)$ whose Seifert form realizes $g_{alg}(P(U))$. Using general position, we can and do assume that $\eta$ intersects $F$ transversely so that the intersection of a small enough tubular neighborhood $N(\eta)$ of $\eta$ and the surface $F$ consists of a collection of $k$ disjoint disks. Denote by $p$ and $n$ the number of disks that intersect $\eta$ positively and negatively, respectively, and note that $w=p-n$.
To prove the lemma it is enough to modify $F$ such that $k=w$, or equivalently that $n=0$, without losing the property that its Seifert form realizes $g_{alg}(P(U))$. This can be achieved by stabilizations, which we prove in detail in the following paragraph.

Assume that $n>0$. Choose a disk $D_i^-\subset F$ intersecting $\eta$ negatively and a disk $D_i^+\subset F$ intersecting $\eta$ positively that are adjacent on $\eta$ (i.e.~they are connected by an arc on $\eta$ that is disjoint from all the other disks). 
\begin{figure}[h!]
\def\svgwidth{0.125\textwidth}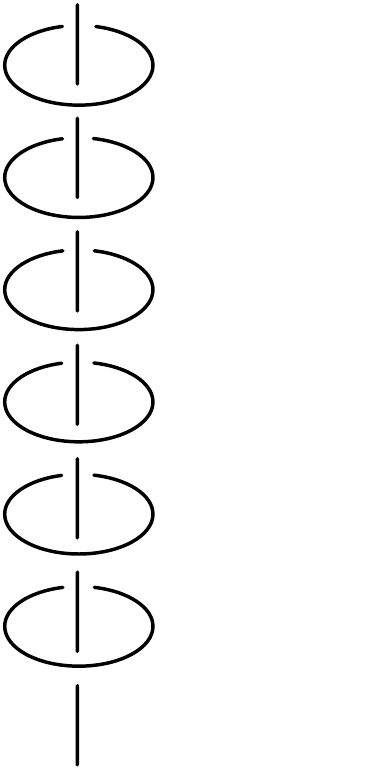\hspace{2cm}
\def\svgwidth{0.125\textwidth}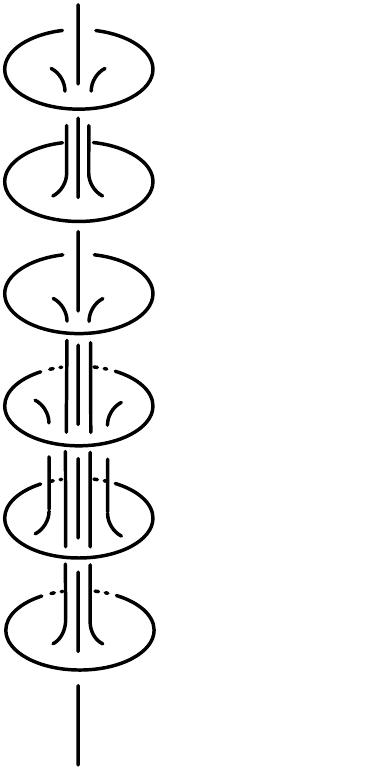
\caption{The unknotted component $\eta$ and some disks in $F\cap N(\eta)$ (left) and the annulus obtained after stabilizing (right).}\label{local-stab}
\end{figure}
Let $a$ be an arc in $\eta$ joining $D_i^+$ to $D_i^-$ such that $a$ is disjoint from all the other disks. Stabilize $F$ using a tube surrounding the arc $a$ to find a new Seifert surface that has two fewer intersections with $\eta$. Iterate this procedure of choosing two disks and stabilizing until a total of $n$ stabilizations have happend.
Call the result of these stabilizations $F'$ and notice that $F'$ intersects $\eta$ only with positive sign, and so if $k'$ denotes the number of disks in the intersection $F'\cap N(\eta)$, then $k'=w$ as sought. For a local picture of this procedure see \Cref{local-stab}. Finally, \cite[Lemma 14]{FellerLewark_16} shows that stabilization of a Seifert surface preserves the property of realizing $g_{alg}$ and so $F'$ also realizes $g_{alg}(P(U))$. We then let $G=F'\cap V$.
\end{proof}

With the previous lemma in place, we are now ready to prove \Cref{prop:maininequality}.

\begin{proof}[Proof of \Cref{prop:maininequality}] Fix a knot $K$ and a pattern $P$ with $r\geq1$ components and algebraic winding number $w$. Without loss of generality assume $w\geq0$, and let $G$ be a Seifert surface for $P \sqcup w l$ as in \Cref{lemma:NiceSSforP}, and let $V_1$ be a Seifert matrix for $P(U)$ corresponding to a choice of a basis for the first homology of $G \cup w D^2$. Similarly, let $S$ be a Seifert surface for $K$ that realizes $g_{alg}(K)$ and let $V_2$ be a Seifert matrix corresponding to a choice of a basis for the first homology of $S$. We can and do assume that we have picked our bases for the first homology of $G \cup w D^2$ and $S$ such that

$$V_1=\left[\begin{array}{cc}A_1& * \\ * & *\end{array}\right], \qquad\text{ and }\qquad V_2=\left[\begin{array}{ccc}A_2 & \begin{array}{cc}B & C\end{array} \\ \
  \begin{array}{c} B^T \\C^T \end{array} & D \end{array}\right], $$
where the matrices $V_1$ and $V_2$ are of size $(2m_1+r-1)\times (2m_1+r-1)$ and $2m_2\times 2m_2$, respectively, for some non-negative integers $m_1,m_2$, and for $i=1,2$ the matrix $A_i$ is an Alexander trivial matrix of size $2(m_i-g_i)\times 2(m_i-g_i)$ for $g_1=\galg(P(U))$ and $g_2= g_{alg}(K)$. We note that $B$ and $C$ are $2(m_2-g_2) \times g_2$ matrices, and  we may further choose our basis for $H_1(S)$ such that $D= \left[ \begin{array}{cc} D_{11}  & D_{12} \\ D_{21} & D_{22} \end{array} \right]$ is a $2g_2 \times 2g_2$ matrix such that $D-D^T= \left[ \begin{array}{cc}
0 & I_{g_2} \\
-I_{g_2} & 0
\end{array}\right]$.

Let $G(K)$ denote $h(G)$ for $h:V\to N(K)$ as in \Cref{def:sat}, in other words the result of cabling $G$ into $K$. Let $\widetilde{F}$ to be the Seifert surface of $P(K)$ given as
\[\widetilde{F}=G(K)\cup w{F},\] where as usual $wF$ denotes $|w|$ many parallel copies of $F$ with boundaries equal to the boundaries of $G(K)$ and $\widetilde{F}=G(K)\cup w{F}$ gets the orientation induced by $G(K)$. Then, pushing forward the basis of $H_1(G\cup w D,\Z)$ via $h_*$ and taking parallel copies of the basis of $H_1(F,\Z)$ chosen earlier we obtain a basis for $H_1(\widetilde{F};\Z)$ and the following Seifert matrix for $P(K)$:

\[V=\left[\begin{array}{c|c}V_1 & 0 \\[0.5ex]\hline
0 & |w|V_2
\end{array}\right], \text{ where } |w|V_2:= \left[ \begin{array}{cccc}
V_2&V_2&\cdots& V_2\\
V^T_2&V_2&\cdots& V_2\\
\vdots&\vdots&\ddots&\vdots\\
V^T_2&V^T_2&\cdots& V_2\\
\end{array}\right] \]
Compare also with the construction in~\cite[Chapter 6]{Lickorish_97}, where this calculation is given for a particular, similarly constructed Seifert surface for $P(K)$. Note that if $|w|$ is $1$ or $0$, then $V$ is a Seifert matrix for $P(U)\sharp K$ and $P(U)$, respectively. This establishes the `in fact'-part of \Cref{prop:maininequality}.

Next, observe that a $2m \times 2m$ Alexander trivial submatrix $M_0$ of a matrix $M$ and a
$2n \times 2n$ Alexander trivial submatrix $N_0$ of a matrix $N$ automatically combine to give a $2(m+n) \times 2(m+n)$ Alexander trivial submatrix  $M_0\oplus N_0$ of $M \oplus N$. Since $V=V_1\oplus |w|V_2$, it therefore suffices to show that there exists a submatrix $X_\Delta$  of $|w|V_2$ that is Alexander trivial and of size $2(|w|m_2-g_2)  \times 2(|w|m_2-g_2)$. To simplify the matrix manipulation, notice that a simple matrix congruence transforms $|w|V_2$ into the matrix {\fontsize{8pt}{8pt}\selectfont
\begin{align*}
X=&\left[\begin{array}{cccc} V_2& 0  & \ldots & 0  \\
-(V_2-V^T_2) &  V_2-V^T_2 & \ldots & 0  \\
0& -(V_2-V^T_2) &\ldots & 0  \\
\vdots & \vdots & \ddots &\vdots \\
0 & 0 & \ldots &  V_2-V^T_2
\end{array}\right]\\=&\left[\begin{array}{ccc ccc c ccc ccc}
 \phantom{1}&&&&&&&&&&&&\\\cline{1-3}
\multicolumn{1}{|c}{A_2} & B & C&\multicolumn{3}{|c}{}&&\multicolumn{6}{c}{}  \\
\multicolumn{1}{|c}{ B^T} & D_{11} & D_{12} &\multicolumn{3}{|c}{}&\cdots&\multicolumn{6}{c}{}\\
\multicolumn{1}{|c}{  C^T} & D_{21} & D_{22} &\multicolumn{3}{|c}{}&&\multicolumn{6}{c}{} \\\cline {1-6}
 \multicolumn{1}{|c}{A_2^T-A_2}& 0 & 0 &\multicolumn{1}{|c}{A_2-A_2^T}& 0 & 0& \multicolumn{1}{|c}{}&\multicolumn{6}{c}{} \\
\multicolumn{1}{|c}{ 0} & 0 &- I_g&\multicolumn{1}{|c}{0} & 0 & I_g&\multicolumn{1}{|c}{\cdots}&\multicolumn{6}{c}{} \\
\multicolumn{1}{|c}{ 0} & I_g &0&\multicolumn{1}{|c}{0} & -I_g &0&\multicolumn{1}{|c}{}&\multicolumn{6}{c}{} \\\cline {1-6}
 \multicolumn{6}{c}{}&\ddots& \multicolumn{6}{c}{}\\ \cline {8-10}
 \multicolumn{6}{c}{}&\multicolumn{1}{c|}{}&A_2-A_2^T& 0 & 0 &\multicolumn{3}{|c}{} \\
\multicolumn{6}{c}{}&\multicolumn{1}{c|}{\cdots}&0 & 0 &I_g&\multicolumn{3}{|c}{} \\
\multicolumn{6}{c}{}&\multicolumn{1}{c|}{}&0 & -I_g &0&\multicolumn{3}{|c}{} \\\cline {8-13}
 \multicolumn{6}{c}{}&\multicolumn{1}{c|}{}&A_2^T-A_2& 0 & 0&\multicolumn{1}{|c}{A_2-A_2^T}& 0 & \multicolumn{1}{c|}{0}  \\
\multicolumn{6}{c}{}&\multicolumn{1}{c|}{\cdots}&0 & 0 &- I_g &\multicolumn{1}{|c}{0} & 0 & \multicolumn{1}{c|}{I_g} \\
\multicolumn{6}{c}{}&\multicolumn{1}{c|}{}&0 & I_g &0 &\multicolumn{1}{|c}{0} & -I_g &\multicolumn{1}{c|}{0} \\\cline {8-13}
 \phantom{1}&&&&&&&&&&&&\\
\end{array}\right].
\end{align*}}%
That is, $|w|V_2$ is congruent to a $|w| \times |w|$ block matrix $X$ with $(i,j)$-block entry given by $V_2$ if $i=j=1$, by $V_2-V_2^T$ if $i=j>1$, by $V_2^T-V_2$ if $i=j+1$, and  $0$ otherwise.
Then, replacing $X$ by $QXQ^t$, where $Q$ is a permutation matrix, we obtain
{\fontsize{8pt}{8pt}\selectfont
\renewcommand\arraystretch{1.125}
\begin{align*}
X'&=
 \left[
 \begin{array}{cccc cc cc c cc cc}
 \phantom{1}&&&&&&&&&&&&\\\cline{1-6}
\multicolumn{4}{|c}{\multirow{4}{*}{$Y$}}&\multicolumn{1}{|c}{B}&\multicolumn{1}{c|}{C}&&&&&&&\\\cline{5-6}
\multicolumn{4}{|c}{}&\multicolumn{1}{|c}{0}&\multicolumn{1}{c|}{0}&&&&&&&\\
\multicolumn{4}{|c}{}&\multicolumn{1}{|c}{\vdots}&\multicolumn{1}{c|}{\vdots}&&&\cdots&&&&\\
\multicolumn{4}{|c}{}&\multicolumn{1}{|c}{0}&\multicolumn{1}{c|}{0}&&&&&&&\\\cline{1-6}
\multicolumn{1}{|c|}{B^T}&0&\cdots&0&\multicolumn{1}{|c}{D_{11}}&\multicolumn{1}{c|}{D_{12}}&& &&&&&\\
\multicolumn{1}{|c|}{C^T}&0&\cdots&0&\multicolumn{1}{|c}{D_{21}}&\multicolumn{1}{c|}{D_{22}}&& &\cdots&&&&\\\cline{1-8}
&&&&\multicolumn{1}{|c}{0}   &\multicolumn{1}{c|}{-I_g}& 0   &\multicolumn{1}{c|}{I_g}&&&&&\\
&&&&\multicolumn{1}{|c}{I_g}&\multicolumn{1}{c|}{0}    &-I_g&\multicolumn{1}{c|}{0}   &\cdots&&&&\\\cline{5-8}
&&& & & &&&\ddots&&&&\\\cline{10-11}
\multicolumn{8}{c}{}&&\multicolumn{1}{|c}{0} & \multicolumn{1}{c|}{I_g} &  & \\
\multicolumn{8}{c}{}&\cdots&\multicolumn{1}{|c}{-I_g}& \multicolumn{1}{c|}{0} &  & \\\cline{10-13}
\multicolumn{8}{c}{}&&\multicolumn{1}{|c}{0} & \multicolumn{1}{c|}{-I_g} & 0 & \multicolumn{1}{c|}{I_g}\\
\multicolumn{8}{c}{}&\cdots&\multicolumn{1}{|c}{I_g} & \multicolumn{1}{c|}{0}& -I_g &\multicolumn{1}{c|}{0} \\\cline{10-13}
 \phantom{1}&&&&&&&&&&&&\\
\end{array}
\right],
\end{align*}}%
where {\fontsize{8pt}{8pt}\selectfont$Y=\left[\begin{array}{cccc} A_2& 0  & \ldots & 0  \\
-(A_2-A^T_2) &  A_2-A^T_2 & \ldots & 0  \\
0& -(A_2-A^T_2) &\ldots & 0  \\
\vdots & \vdots & \ddots &\vdots \\
0 & 0 & \ldots &  A_2-A^T_2
\end{array}\right]$}, i.e.~$Y$ is a $|w| \times |w|$ block matrix with $(i,j)$ block entry equal to $A_2$ if $i=j=1$, $A_2-A_2^T$ if $i=j>1$, $A_2^T-A_2$ if $i=j+1$ and $0$ else.

We will show that $X_\Delta$, the matrix obtained from $X'$ by deleting the first blockrow and column after $Y$ and the last blockrow and column, is Alexander trivial. Indeed, note that the matrix $X_\Delta - t(X_\Delta)^T$ is given by
{\fontsize{7pt}{7pt}\selectfont
\begin{align*}
\left[
 \begin{array}{cccc c cc c cc cc c}
 \phantom{1}&&&&&&&&&&&\\\cline{1-5}
\multicolumn{4}{|c}{\multirow{4}{*}{$Y- tY^T$}}&\multicolumn{1}{|c|}{(1-t)C}&&&&&&&\\\cline{5-5}
\multicolumn{4}{|c}{}&\multicolumn{1}{|c|}{0}&&&&&&&\\
\multicolumn{4}{|c}{}&\multicolumn{1}{|c|}{\vdots}&&&\cdots&&&&\\
\multicolumn{4}{|c}{}&\multicolumn{1}{|c|}{0}&&&&&\\\cline{1-5}
\multicolumn{1}{|c|}{(1-t)C^T}&0&\cdots&0&\multicolumn{1}{|c|}{D_{22}-tD_{22}^T}& tI_g &&\cdots&&&&&\\\cline{1-7}
&&&&\multicolumn{1}{c|}{-I_g}&	0&\multicolumn{1}{c|}{(1+t)I_g}&&&&&&\\
&&&&\multicolumn{1}{c|}{}	 &-(1+t)I_g&\multicolumn{1}{c|}{0}&\cdots&&&&&\\\cline{6-7}
&&&&&&&\ddots&&&&&\\\cline{9-12}
&& &&&&&&\multicolumn{1}{|c}{0}&\multicolumn{1}{c|}{(1+t)I_g} &0&\multicolumn{1}{c|}{-tI_g}&0\\
&& &&&&&\cdots&\multicolumn{1}{|c}{-(1+t)I_g}&\multicolumn{1}{c|}{0}&tI_g&\multicolumn{1}{c|}{0}&0\\\cline{9-12}
&&&&&&&&\multicolumn{1}{|c}{0}&-I_g&\multicolumn{1}{|c}{0}&\multicolumn{1}{c|}{(1+t)I_g} &0\\
&&&&&&&\cdots&\multicolumn{1}{|c}{I_g}&0&\multicolumn{1}{|c}{-(1+t)I_g}&\multicolumn{1}{c|}{0}&\multicolumn{1}{C}{ tI_g }\\\cline{9-12}
&&&&&&&&0&0&0&\multicolumn{1}{C}{ -I_g} &0 \\
\end{array}
\right]
\end{align*}}%
and so the only nonzero entry in its final block row is $-I_g$ in the penultimate block column, and similarly the only nonzero entry in its final block column is $tI_g$ in the penultimate block row. We can therefore delete the final two rows and columns of $X_\Delta - t(X_\Delta)^T$ without changing its determinant. Thus, $\det\left(X_\Delta - t(X_\Delta)^T\right)$ is given by {\fontsize{8pt}{8pt}\selectfont
\begin{align*}
&\det \left[
 \begin{array}{cccc c cc c cc c}
 \phantom{1}&&&&&&&&&\\\cline{1-5}
\multicolumn{4}{|c}{\multirow{4}{*}{$Y- tY^T$}}&\multicolumn{1}{|c|}{(1-t)C}&&&&&\\\cline{5-5}
\multicolumn{4}{|c}{}&\multicolumn{1}{|c|}{0}&&&&&\\
\multicolumn{4}{|c}{}&\multicolumn{1}{|c|}{\vdots}&&&\cdots&&\\
\multicolumn{4}{|c}{}&\multicolumn{1}{|c|}{0}&&&\\\cline{1-5}
\multicolumn{1}{|c|}{(1-t)C^T}&0&\cdots&0&\multicolumn{1}{|c|}{D_{22}-tD_{22}^T}& tI_g &&\cdots&&&\\\cline{1-7}
&&&&\multicolumn{1}{c|}{-I_g}&	0&\multicolumn{1}{c|}{(1+t)I_g}&&&&\\
&&&&\multicolumn{1}{c|}{}	 &-(1+t)I_g&\multicolumn{1}{c|}{0}&\cdots&&&\\\cline{6-7}
&&&&&&&\ddots&&&\\\cline{9-10}
&&&&&&&&\multicolumn{1}{|c}{0}&\multicolumn{1}{c|}{(1+t)I_g} &0\\
&&&&&&&\cdots&\multicolumn{1}{|c}{-(1+t)I_g}&\multicolumn{1}{c|}{0}&\multicolumn{1}{C}{ tI_g }\\\cline{9-10}
&&&&&&&\cdots&0&\multicolumn{1}{C}{ -I_g}&0 \\
\end{array}
\right]
\intertext{\normalsize and repeating this procedure one observes that}
&\det\left(X_\Delta - t(X_\Delta)^T\right)=
\det \left[
 \begin{array}{cccc c c}
  \phantom{1}&&&&&\\\cline{1-5}
 \multicolumn{4}{|c}{\multirow{4}{*}{$Y- tY^T$}}&\multicolumn{1}{|c|}{(1-t)C}&0\\\cline{5-5}
\multicolumn{4}{|c}{}&\multicolumn{1}{|c|}{0}&\\
\multicolumn{4}{|c}{}&\multicolumn{1}{|c|}{\vdots}&\vdots\\
\multicolumn{4}{|c}{}&\multicolumn{1}{|c|}{0}&\\\cline{1-5}
\multicolumn{1}{|c|}{(1-t)C^T}&0&\cdots&0&\multicolumn{1}{|c|}{D_{22}-tD_{22}^T}& tI_g\\\cline{1-5}
0&&\cdots&&\multicolumn{1}{c}{-I_g}&	0\\
\end{array}
\right]=\det \left(Y- tY^T\right)
\end{align*}}%

By reversing the row and column moves we performed on $|w|V_2$  at the beginning of this argument we see that $Y$ is congruent to $|w|A_2$, and hence
\[\det\left(X_\Delta - t(X_\Delta)^T\right)= \det(Y- tY^T) = \det(|w|A_2- t(|w|A_2)^T).\]
To see that $|w|A_2$ is Alexander trivial notice that if $J$ is a knot with Seifert form $A_2$, then $|w|A_2$ is a Seifert form for $C_{|w|,1}(J)$. Then Litherland's formula of \Cref{eq:alexsatellite} implies that
 \[\det(|w|A_2- t(|w|A_2)^T)=\Delta_{C_{|w|,1}(J)}(t)= \Delta_J(t^{|w|})=1.\qedhere \]
 \end{proof}

To end this section, we include an example that illustrates that the inequality from \Cref{prop:maininequality} can be sharp and moreover, can sometimes be attained in a nice geometric way.

\begin{Example}[The Mazur pattern]\label{ex:mazur}
The Mazur satellite of the figure-eight knot, $M(4_1)$, has a genus 2 Seifert surface $F$ constructed in \Cref{fig:seifertexample} from two genus 1 surfaces realizing the algebraic genera of $M(U)$ and of $4_1$, respectively.  The proof of~\Cref{prop:maininequality} implies that there is some curve $\gamma$ which bounds a genus 1 subsurface of $F$ and, when considered as a knot in $S^3$, has $\Delta_{\gamma}(t)=1$.
In fact, as illustrated in \Cref{Mazur}, we can pick $\gamma$ to be isotopic to the positive Whitehead double $D(4_1)$.

\begin{figure}[h]\label{fig:mazur41curve}
\includegraphics[height=4cm]
{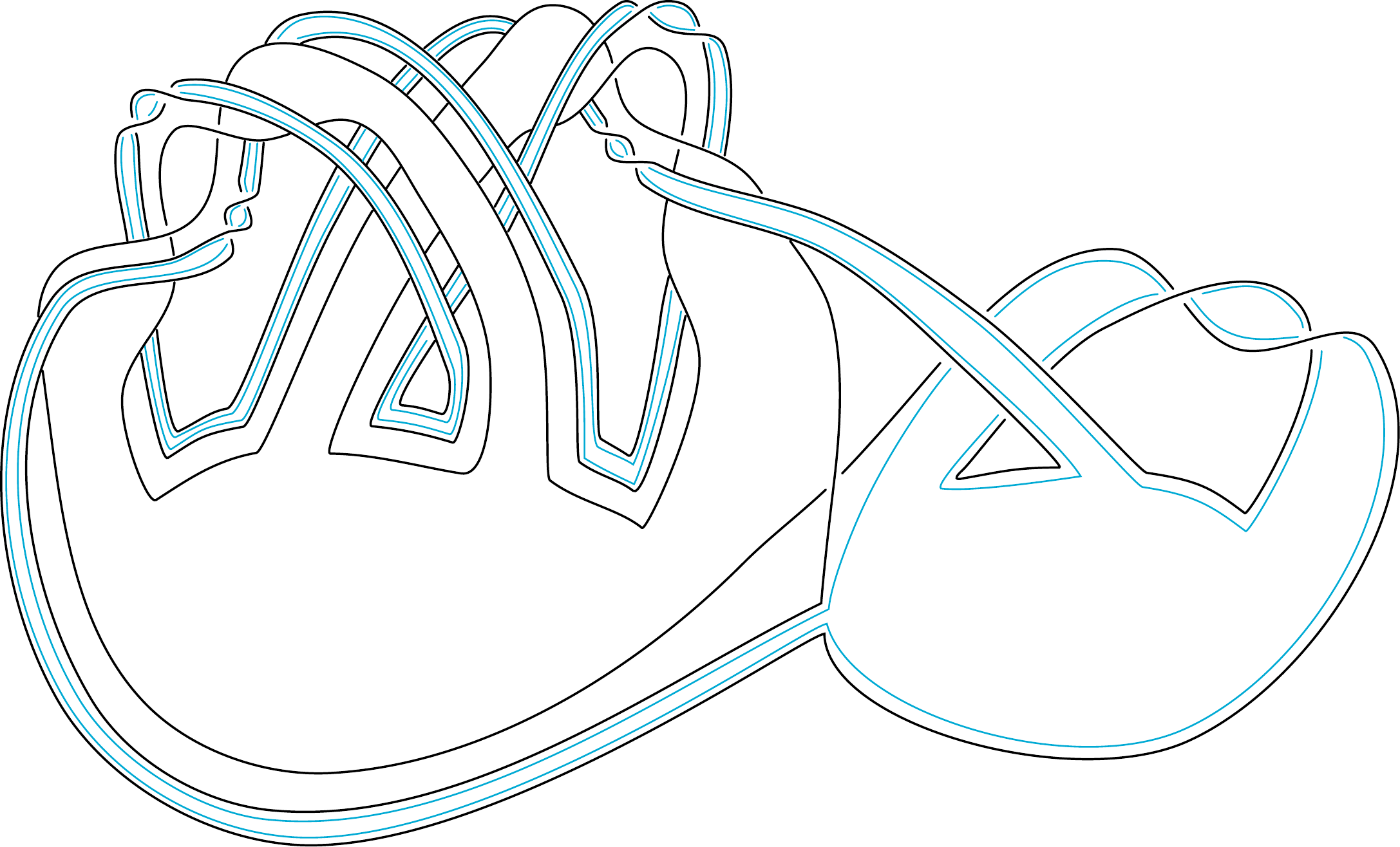}
\caption{A Seifert surface for $M(4_1)$ with separating curve $\gamma$ isotopic to $D(4_1)$.}\label{Mazur}
\end{figure}
\end{Example}

\section{Lower bounds on $\g$ and satellite operations}\label{sec:CG}

In this section, we discuss lower bounds for the topological 4-genera of knots, namely Tristram-Levine signatures and Casson-Gordon signatures, and explain why these invariants cannot be used to disprove~\Cref{conj:(1)holds}. While this is immediate from classical formulas in the case of Tristram-Levine signatures, we consider it a priori somewhat surprising that Casson-Gordon signatures fail to disprove \Cref{conj:(1)holds}. All patterns $P$ in this section are connected, i.e.~they are knots in a solid torus $V$.\\

The Tristram-Levine signatures $\sigma_\omega$ are classical knot invariants~\cite{tristram,levine69}, which have a simple behavior with respect to satellite operations and provide lower bounds for $\g$. Namely, for a pattern $P$ with winding number $w$ one has
\begin{equation}\label{eq:sigformulaforsat}
 \sigma_\omega(P(K))=\sigma_\omega(P(U))+\sigma_{\omega^w}(K) \text{ for all knots $K$ and $\omega\in S^1$; see~\cite{litherland-it}.}\end{equation}
A classical result establishes that signatures give a lower bound for $\g$:
\begin{equation}\label{eq:siglowerboundg4}|\sigma_\omega(K)|\leq2\g(K)\text{ for all knots $K$ and regular $\omega\in S^1$ \cite{taylor,livingston-genus}.}\end{equation}
Here, $\omega\in S^1$ is said to be regular  if it does not arise as the root of an Alexander polynomial of a knot. For example, all prime-power order roots of unity are regular.

As a consequence, one has that
\[\max_{\text{regular }\omega \in S^1}|\sigma_\omega(P(K))|\leq \max_{\text{regular }\omega \in S^1}|\sigma_\omega(P(U))|+\max_{\text{regular }\omega \in S^1}|\sigma_{\omega}(K)|\leq 2\g(P(U)) + 2\g(K),\] which shows the lower bound for $\g(P(K))$ given by the Levine-Tristram signatures of $P(K)$ cannot be used to establish that a pair $P$ and $K$  fails to satisfy the inequality of \Cref{conj:(1)holds}.

The next family of slice genus bounds come from Casson-Gordon signatures by work of Gilmer.
We state the following \Cref{prop:nocg}, our main result of this section, before recalling the relevant background. Informally,
one may paraphrase \Cref{prop:nocg} as `one cannot use Casson-Gordon signatures to prove $\g(P(K))> \g(P(U))+\g(K)$'.

\begin{thm}\label{prop:nocg}
Let $P$ be a pattern and $K$ be any knot. Then $P(K)$ satisfies the Gilmer bounds for $g\geq \g(P(U))+\g(K)$. That is, for any prime power $n$, there is a decomposition of $H_1(\Sig_n(K))$ as described in \Cref{thm:Gilmer} below.
\end{thm}

\subsection*{Casson Gordon ala Gilmer}

We will be working with torsion abelian groups $G$ equipped with linking forms $\lambda \colon G \times G \to \mathbb{Q}/ \mathbb{Z}$. In particular, when we write $G \cong G_1 \oplus G_2$  we are implicitly decomposing the pair $(G, \lambda) \cong (G_1, \lambda_1) \oplus (G_2, \lambda_2)$.
Our main examples of such pairs $(G, \lambda)$ will be $G=H_1(\Sigma_n(K))$, the first homology of the $n$th cyclic branched cover of a knot $K$ for $n$ a prime power,
and $\lambda= \lambda_n^K$ the so-called torsion linking form.

\begin{definition}
Given a subgroup $G \leq H_1(\Sig_n(K))$, we call $H \leq G$ an \emph{invariant metabolizer} of $G$ if
\begin{itemize}
\item $H$ is a metabolizer for $\lambda_n|_G$, i.e.~$|H|^2=|G|$ and $\lambda_n|_{H \times H}=0$.\footnote{We warn the reader that the traditional definition of a metabolizer $M$ of $G$, i.e.~a subgroup satisfying
\[M= M^\perp:= \{ g \in G \text{ : } \lambda_n(g, m)=0 \text{ for all }  m \in M\}\]
coincides with this definition only when $\lambda_n|_{G \times G}$ is nonsingular.}

\item $H$ is preserved by the $\Z_n$-action induced by the covering transformation of $\Sigma_n(K)$.
\end{itemize}
\end{definition}

To a knot $K$, a prime power $n$, and a prime power order character $\chi \colon H_1(\Sigma_n(K)) \to \Z_q$, Casson and Gordon associate a collection of rational numbers  $\{\sigma_r \tau(K, \chi)\}_{r=1}^{q}$ called Casson-Gordon signatures~\cite{cg-slice,cg-cob}.
These signatures were employed to give the first examples of non-slice yet algebraically slice knots. Work of Gilmer extended the sliceness obstruction of~\cite{cg-slice,cg-cob} to give lower bounds on $\g$~\cite{Gilmer_slicegenus}, stated here in the reformulation and mild strengthening of~\cite{Millerwinding}.
From now on, for  $n \in \mathbb{N}$ we fix a primitive $n$th root of unity denoted by $\omega_n$.

\begin{thm}[\cite{Gilmer_slicegenus, Millerwinding}] \label{thm:Gilmer}
Let $K$ be a knot and suppose that $\g(K) \leq g$.
Then for any prime power $n$ there is a decomposition of $H_1(\Sig_n(K))= A_1 \oplus A_2$ so that the following properties hold:
\begin{enumerate}[label=(\Roman*)]
\item\label{itemI}  $A_1$ has an even presentation of rank $2(n-1)g$ with signature equal to $\sum_{i=1}^n \sig_K(\om_n^i)$.
\item\label{itemII}  $A_2$ has an invariant metabolizer $B$ such that given any prime power order character $\chi$ which vanishes on $A_1 \oplus B$, we have   \[| \sig_{1} \tau(K, \chi) + \sum_{i=1}^n \sig_K(\om_n^i) | \leq 2ng.
\]
\item\label{itemIII} $A_1 \oplus B$ is also covering transformation invariant.
\end{enumerate}
\end{thm}
Observe that an equivalent formulation of \Cref{thm:Gilmer} states that for $\chi$ an order $q$ character as above we have $|\sig_{r} \tau(K, \chi) + \sum_{i=1}^n \sig_K(\om_n^i) | \leq 2ng$ for any $r=1, \dots q$, since
$\sig_{r}\tau(K, \chi)= \sig_1 \tau(K, r' \chi)$ for some $r'$ and $\chi|_{H}=0$ implies that $r'\chi|_H=0$ as well.
\\

Given a knot $K$ and some $g \geq 0$, we say that $(K,n,g)$ satisfies the Gilmer 4-genus bounds if the conclusions of  \Cref{thm:Gilmer} hold. If $(K,n,g)$ satisfies the Gilmer bound for all prime powers $n$, we say that $(K,g)$ satisfies the Gilmer bound.

\subsection*{Casson-Gordon signatures of a satellite knot}
We will need the following general formula for the Casson-Gordon signatures of a satellite knot.
Recall that given a map $\chi\colon H_1(\Sig_n(K)) \to \Z_q$, we denote by $\sig_{r}\tau(K, \chi)$ the $r$th Casson-Gordon signature of $(K, \chi)$.
In the exceptional case when $n=1$ and so $\Sig_1(K)=S^3$ and $\chi$ must be trivial, we somewhat abusively let $\sig_{r}\tau(K,\chi)$ denote the Tristram-Levine signature $\sigma_K(\omega_q^r)$.

\begin{thm}[Litherland] \label{thm:litherland}
Let $P$ be a pattern described by a curve $\eta$ in the complement of $P(U)$, i.e.~the solid torus $V$ is $S^3 \smallsetminus \nu(\eta)$.
 Suppose  $P$ has winding number $m$ and let $n\in \N$.
Let $d= \gcd(m,n)$.  Then there is a canonical covering transformation invariant isomorphism
\begin{align*}
 \alpha \colon H_1(\Sig_n(P(K))) \to H_1(\Sig_n(P(U))) \oplus \bigoplus_{i=1}^d H_1(\Sig_{n/d}(K)).
 \end{align*}
Supposing also now that $n$ and $q$ are prime powers, let
\[ \chi=(\chi_0, \chi_1, \dots, \chi_d) \colon H_1(\Sig_n(P(U))) \oplus \bigoplus_{i=1}^d H_1(\Sig_{n/d}(K)) \to \Z_q. \]
 Let the homology classes of the $d$ lifts of $\eta$ to $\Sigma_n(P(U))$ be denoted by $\eta_1, \dots \eta_d$.
Then the Casson-Gordon signature $\sigma_1\tau(P(K), \chi \circ \alpha)$ is given by
\begin{align*}
\sigma_1 \tau(P(K), \chi \circ \alpha)=
 \sigma_1\tau(P(U), \chi_0)+ \sum_{i=1}^d \sigma_{\chi_0(\eta_i)} \tau(K, \chi_i).
\end{align*}
\end{thm}

\subsection*{Proof of~\Cref{prop:nocg}}
We now use Litherland's formula for Casson-Gordan sigantures and Gilmer's bounds for $P(U)$ and $K$ to prove~\Cref{prop:nocg}.
\begin{proof}[Proof of~\Cref{prop:nocg}]
Let $g_K= \g(K)$, $g_P= \g(P(U))$, and let $n$ be an arbitrary prime power. We show that $(P(K), n, g)$ satisfies the Gilmer bounds for $g \geq g_P+g_K$.

By \Cref{thm:Gilmer} there is a decomposition of $H_1(\Sigma_n(P(U))) = A^P_1 \oplus A^P_2$ with the following properties:
\begin{enumerate}[label=(P\Roman*)]
\item\label{itemPI} $A^P_1$ has an even rank $2(n-1)g_P$ presentation of signature $\sum_{i=1}^{n} \sig_{P(U)}(\omega_{n}^i)$.
\item\label{itemPII} $A^P_2$ has an invariant metabolizer $B^P$ such that if $\chi \colon H_1(\Sig_{n}(P(U))) \to \Z_q$ is a character of prime power order vanishing on $A^P_1 \oplus B^P$, then
\[\left|\sig_1\tau(P(U), \chi) + \sum_{i=1}^{n} \sig_{P(U)}(\omega_{n}^i)\right| \leq 2ng_P.
\]
\item\label{itemPIII} $A^P_1 \oplus B^P$ is also covering transformation invariant.
\end{enumerate}

Write the algebraic winding number of $P$ as $m=p^a m'$, where $p^a= \gcd(m, n)$. So $n=p^b$ for $b \geq a \geq 0$. Note that when $a=b$, i.e.~$n=p^a$ divides $m$, we have that $\eta$ lifts to $n$ distinct curves in $\Sigma_n(P(U))$ and when $a<b$ we have that $\eta$ lifts to strictly fewer than $n$ curves in $\Sigma_n(P(U))$.\\

{\bf Case 1}: $a=b$, so $n=p^a$ divides $m$.

Decompose $H_1(\Sig_n(P(K)))\cong H_1(\Sig_n(P(U))) \oplus 0$ using $\alpha$ from~\Cref{thm:litherland} and take $A_1=\alpha^{-1}(A^P_1)$, $A_2=\alpha^{-1}(A^P_2)$, and $B=\alpha^{-1}(B^P)$.

To check~\ref{itemI}, we observe that

\begin{equation}\label{eq:sumsigmaP(K)b=a}
\sum_{i=1}^n \sigma_{P(K)}(\omega_n^i)\overset{\text{\eqref{eq:sigformulaforsat}}}{=}
\sum_{i=1}^n \left( \sigma_{P(U)}(\omega_n^i) + \sigma_K(\omega_n^{im}) \right)
= \sum_{i=1}^n \left(\sigma_{P(U)}(\omega_n^i)+ \sigma_K(\omega_n^{inm'})\right)
=\sum_{i=1}^n \left(\sigma_{P(U)}(\omega_n^i)\right)
.
\end{equation}
Thus, $A_1$ has an even presentation of rank $2(n-1)g_K$ with signature $\sum_{i=1}^n \sigma_{P(K)}(\omega_n^i)$ by (PI). Noting that the trivial group $0$ certainly has an even presentation of rank $2(n-1)(g-g_K)$ and signature 0, we have that $A_1$ has an even presentation of rank $2(n-1)g_K+2(n-1)(g-g_K)$ with signature $\sum_{i=1}^n \sigma_{P(K)}(\omega_n^i)+0$. This concludes the proof of~\ref{itemI}.

To check~\ref{itemII}, we calculate that, given any $\chi: H_1(\Sig_n(P(U))) \to \Z_q$ of prime power order with $\chi|_{A_1\oplus B}=0$, we have
\begin{align*}
\begin{array}{rcl}
\left|\sig_1 \tau(P(K), \chi \circ \alpha) +\sum_{i=1}^n\sigma_{P(K)}(\omega_n^i)\right|
&\overset{\text{\eqref{eq:sumsigmaP(K)b=a}}}{=}
& \left|\sig_1 \tau(P(K), \chi \circ \alpha)+\sum_{i=1}^n \left(\sigma_{P(U)}(\omega_n^i)\right)\right|\\
&\overset{\text{\ref{thm:litherland}}}{=}
&\left| \sig_1 \tau(P(U), \chi)+ \sum_{i=1}^n \sig_K(\omega_q^{\chi(\eta_i)})+\sum_{i=1}^n \left(\sigma_{P(U)}(\omega_n^i)\right)\right|\\
&\leq &\left| \sig_1 \tau(P(U), \chi)+\sum_{i=1}^n \left(\sigma_{P(U)}(\omega_n^i)\right)\right|+\left|\sum_{i=1}^n \sig_K(\omega_q^{\chi(\eta_i)})\right|\\
&\overset{\text{\ref{itemPII},\eqref{eq:siglowerboundg4}}}{\leq}
&2ng_P+2ng_K \\
& \leq &2ng.
\end{array}
\end{align*}

Finally, \ref{itemIII} is immediate from \ref{itemPIII} and the covering transformation invariance of $\alpha$.
\\

{\bf Case 2}: $b>a$.

By \Cref{thm:Gilmer} there is a decomposition of $H_1(\Sig_{p^{b-a}}(K)) = A^K_1 \oplus A^K_2$ with the following properties:
\begin{enumerate}[label=(K\Roman*)]
\item\label{itemKI}$A^K_1$ has an even rank $2(p^{b-a}-1)g_K$ presentation of signature $s= \sum_{i=1}^{p^{b-a}} \sig_K(\omega_{p^{b-a}}^i)$.
\item\label{itemKII} $A^K_2$ has an invariant metabolizer $B^K$ such that if $\chi \colon H_1(\Sig_{p^{b-a}}(K)) \to \Z_q$ is a character of prime power order $q$ vanishing on $A_1 \oplus B$, then
\[\left|\sig_1\tau(K, \chi) + s\right| \leq 2(p^{b-a})g_K.
\]
\item\label{itemKIII} $A_1 \oplus B$ is also covering transformation invariant.
\end{enumerate}

Decompose $H_1(\Sig_n(P(K)) \cong H_1(\Sig_n(P(U))) \oplus \bigoplus_{i=1}^{p^a} H_1(\Sig_{p^{b-a}}(K))$ using $\alpha$ from~\Cref{thm:litherland} and take \[A_1=\alpha^{-1}\left(A_1^P \oplus \bigoplus_{i=1}^{p^a} A^K_1 \right) \et A_2=\alpha^{-1}\left(A_2^P \oplus \bigoplus_{i=1}^{p^a} A_2^K\right).\]
Observe that by taking the direct sum of our assumed presentations for $A_1^P$ and $A_1^K$ from \ref{itemPI} and~\ref{itemKI}, respectively, we have that  $A_1\cong A_1^P \oplus \bigoplus_{i=1}^{p^a} A_1^K$ has an even presentation of rank
\[2(p^b-1)g_P+p^a 2(p^{b-a}-1)g_K=2(p^b-1)g_P+ 2(p^b-p^a)g_K \leq 2(p^b-1)g=2(n-1)g\]
and  signature $\sum_{i=1}^n \sigma_{P(U)}(\omega_n^i)+p^a s$.
However, since $p^a=\gcd(p^b, p^a m')$ we know that $(p, m')=1$ and so $\{\omega_{p^{b-a}}^{m' j}: j=1, \dots, p^{b-a}\}=\{\omega_{p^{b-a}}^{i}: i=1, \dots, p^{b-a}\}$. It follows that
\begin{align*}
p^a  s&= p^a \sum_{i=1}^{p^{b-a}} \sig_K(\omega_{p^{b-a}}^i)
= p^a \sum_{j=1}^{p^{b-a}} \sig_K(\omega_{p^{b-a}}^{m'j})
= \sum_{j=1}^{p^b} \sig_K(\omega_{p^{b-a}}^{m'j})
= \sum_{j=1}^{p^b} \sig_K(\omega_{p^{b}}^{p^am'j})=\sum_{i=1}^{n} \sig_K(\omega_{n}^{mi})
\end{align*}
and thus \begin{equation}\label{eq:sumsig(P(K))b>a}
\sum_{i=1}^n \sigma_{P(U)}(\omega_n^i)+p^a  s= \sum_{i=1}^{n}\left(\sig_{P(U)}(\omega_{n}^i)+ \sig_K(\omega_{n}^{mi})\right)= \sum_{i=1}^{n} \sig_{P(K)}(\omega_{n}^i).\end{equation}
This concludes the proof of \ref{itemI} since the even presentation of $A_1$ of rank $2(p^b-1)g_P+ 2(p^b-p^a)g_K$ and signature $\sum_{i=1}^{n} \sig_{P(K)}(\omega_{n}^i)$ just described can be increased if necessary to have rank $2(n-1)g=2(p^b-1)g$ by connect sum with an even presentation of the trivial group with signature $0$ and appropriate rank.

Now, note that $B^P \oplus \bigoplus_{i=1}^{p^a} B^K$ is an invariant metabolizer for $A_1^P \oplus \bigoplus_{i=1}^{p^a} A_2^K$ and set \[B=\alpha^{-1}\left(B^P\oplus\bigoplus_{i=1}^{p^a} B^K\right).\]
We further note that $A_1\oplus B$ is covering transformation invariant by the covering transformation invariance of $\alpha$
and the fact that
\[
\alpha(A_1\oplus B)=\alpha(A_1)\oplus \alpha(B)= \left(A_1^P \oplus \bigoplus_{i=1}^{p^a} A_1^K \right) \oplus \left(B^P\oplus\bigoplus_{i=1}^{p^a} B^K\right)= B^P \oplus \bigoplus_{i=1}^{p^a} (A_1^K \oplus B^K)
\]
is covering transformation invariant by~\ref{itemPIII},~\ref{itemKIII} , which establishes~\ref{itemIII}.

To check~\ref{itemII}, let
\[\chi= (\chi_0, \chi_1, \dots, \chi_{p^a}): H_1(\Sig_n(P(U))) \oplus \bigoplus_{i=1}^{p^a} H_1(\Sig_{p^{b-a}}(K)) \to \Z_q\] be a character of prime power order, and suppose that $\chi$ vanishes on
\[{\alpha(A_1\oplus B)= (A_1^P \oplus B^P) \oplus \bigoplus_{i=1}^{p^a} (A_1^K \oplus B^K)}.\]
 In particular, $\chi_0$ vanishes on $A_1^P\oplus B^P$ and $\chi_i$ vanishes on the $i$th copy of $A_1^K\oplus B^K$.
Now observe that
\begin{align*}
\left|\sig_1 \tau(P(K), \chi \circ \alpha)+ \sum_{i=1}^{p^b} \sig_{P(K)}(\omega_{p^b}^i) \right|\underset{\eqref{eq:sumsig(P(K))b>a}}{\overset{\text{\ref{thm:litherland}}}{=}} &\left| \sig_1 \tau(P(U), \chi_0)+ \sum_{i=1}^{p^a} \sig_{\chi_0(\eta_i)} \tau(K, \chi_i) +\sum_{i=1}^n \sigma_{P(U)}(\omega_n^i)+p^a s\right|\\
\leq&\left|\sig_1 \tau(P(U), \chi_0)+\sum_{i=1}^n \sigma_{P(U)}(\omega_n^i)\right|+\left| \sum_{i=1}^{p^a} \left(\sig_{\chi_0(\eta_i)} \tau(K, \chi_i) +s\right)\right|\\
\leq&\left|\sig_1 \tau(P(U), \chi_0)+\sum_{i=1}^n \sigma_{P(U)}(\omega_n^i)\right|+\sum_{i=1}^{p^a} \left| \sig_{\chi_0(\eta_i)} \tau(K, \chi_i) +s\right|  \\
\underset{\text{\ref{itemKII}}}{\overset{\text{\ref{itemPII}}}{\leq}}& 2ng_P+\sum_{i=1}^{p^a} 2p^{b-a}g_K=2ng_P+2ng_K\leq 2ng. \qedhere
\end{align*}

\end{proof}

We remark that, besides Tristram-Levine signatures and Gilmer's Casson-Gordon obstruction, the only known obstruction  to being a  knot with small $\g$ comes from recent work of Cha-Miller-Powell~\cite{cmp}. This work uses certain $L^{(2)}$ $\rho$-invariants to show that certain families of knots with vanishing Tristram-Levine signature functions and vanishing Casson-Gordon sliceness obstructions still have  members with arbitrarily large $\g$. Moreover, their constructions are all of the form $J=\#_{i=1}^n P(K_i)$ for $P$ a winding number 0 satellite with $P(U)$ slice.
However, these techniques  only show  $\g(J)\geq g$ for $g$ orders of magnitude smaller than $\sum_{i=1}^n \g(K_i)$, and hence seem ill-suited to trying to disprove \Cref{conj:(1)holds}.

\section{Contrast with the smooth setting}\label{sec:Exs}

We will use the following result of Hom~\cite{hom} on how the Heegaard Floer invariant $\tau$ behaves under cabling.

\begin{thm}[Hom]\label{rem:tauqp}
Let $K$ be a knot with $\gsm(K)= \tau(K)>0$ then for any $w>0$ we have
\[ \gsm(C_{w,1}(K))= \tau(C_{w,1}(K))= w \tau(K) = w \gsm(K)\]
and $\epsilon(C_{w,1}(K))=\epsilon(K)=+1$.
\end{thm}

\begin{prop}\label{prop:smoothcontrast} Let $P$ be a winding number $w$ pattern. Then
$\displaystyle \lim_{n \to \infty} \frac{ \gsm(P(T_{2,2n+1}))}{\gsm(T_{2,2n+1})}=|w|.$
\end{prop}

\begin{proof}
Let $P$ be  a winding number $w$ pattern. Since $P$ and $C_{w,1}$ are homologous in $V$, there exists a surface $F$ in $V \times I$ with boundary $P \times \{1\} \sqcup -C_{w,1} \times \{0\}$. One can use an argument analogous to the one that shows that patterns have a well defined action on concordance,  see Cochran-Harvey~\cite{CHmetrics}, to show that  for any knot $K$,
\[|\gsm(P(K))- \gsm(C_{w,1}(K))| \leq \gsm(P(K) \#-C_{w,1}(K)) \leq g(F).
\]
Therefore, since $\lim_{n \to \infty} \gsm(T_{2,2n+1})= \infty$, we have as desired that
\[
\lim_{n \to \infty}\frac{ \gsm(P(T_{2,2n+1}))}{\gsm(T_{2,2n+1})} = \lim_{n \to \infty}\frac{ \gsm(C_{w,1}(T_{2,2n+1}))}{\gsm(T_{2,2n+1})}=
\lim_{n \to \infty}\frac{wn}{n}=
w.   \qedhere
\]
\end{proof}

\begin{remark} The above argument shows that for any collection $\{K_n\}$ of quasipositive knots (or knots with $\tau(K_n)= \gsm(K_n)\neq 0$ and $\epsilon(K_n)=+1$)  with $\lim_{n \to \infty} g_4(K_n)= \infty$ we have $\displaystyle \lim_{n \to \infty} \frac{ \gsm(P(K_n))}{\gsm(K_n)}=|w|.$
\end{remark}

The following result, together with \Cref{prop:maininequality} in the winding number 0 case,
immediately implies \Cref{cor:limitis1intop}, since $\lim_{n \to \infty} \g(T_{2,2n+1})= \infty$.

\begin{prop}\label{prop:topgenussatellitetorus}
Let $P$ be a winding number $w>0$ pattern.
Then \[ -\g(P(U)) \leq \g(P(T_{2,2n+1}))-\g(T_{2,2n+1}) \leq  \gZ(P(U))\]
\end{prop}

\begin{proof}
Let $K_n= T_{2,2n+1}$.
We first observe that for $t_n \in (\frac{2n-1}{2n+1} \pi,\frac{2n+3}{2n+1}\pi)$,
 we have
\[2n=| \sigma_{e^{i t_n}}(K_n)| \leq 2  \g(K_n)\leq 2 \gZ(K_n) \leq 2 g_3(K_n) = 2n,\]
and hence we have equality throughout.

Now let  $P$ be a pattern of winding number $w>0$ and observe by \Cref{thm:mainIntro} that
\[\g(P(K_n)) \leq \gZ(P(K_n)) \leq \gZ(P(U))+ \gZ(K_n)= \gZ(P(U))+\g(K_n).
\]

We now need  to obtain our lower bound on $\g(P(K_n))$.  Let $s_n \in  \left(\frac{(2n-1) \pi}{(2n+1)w},\frac{(2n+3)\pi}{(2n+1)w}\right)$ be such that $e^{is_n}$ is not a root of $\Delta_{P(U)}(t)$. It follows that $e^{is_n}$ is not a root of $\Delta_{P(K_n)}(t)= \Delta_{P(U)}(t)\cdot \Delta_{K_n}(t^w)$ and so
\begin{align*} 2\g(P(K_n))\geq |\sigma_{e^{i s_n}}(P(K_n))|
&=  |\sigma_{e^{i s_n}}(P(U))+ \sigma_{e^{iws_n}}(K_n)| \\
& \geq 2\g(K_n)- |\sigma_{e^{i s_n}}(P(U))| \geq 2\g(K_n)-2\g(P(U)).\qedhere
\end{align*}
\end{proof}

\begin{prop}\label{prop:bigsmg4}
For any $w, m \in \mathbb{N}$, there exists a winding number $w$ pattern $P$ such that for any quasipositive knot $K$,
\[ \gsm(P(K))= \gsm(P(U))+ |w| \gsm(K)+ m.\]
\end{prop}

\begin{proof}
Let $P_{m,w}= Q^{m} \circ C_{w,1}$, where $Q$ denotes the Mazur pattern, $\circ$ denotes pattern composition, and $Q^m$ denotes the $m$-fold composition of $Q$, which is an winding number 1 pattern which geometrically wraps $3^m$ times about the solid torus.
Note that $P_{m,w}$ is a winding number $w$ pattern.
Let $K$ be a quasipositive knot.
By Levine~\cite{levine16}, if $J$ is any knot with $\epsilon(J)=+1$ then $\tau(Q(J))= \tau(J)+1$ and $\epsilon(Q(J))=+1$.
Applying this to $J=C_{w,1}(K)$ and using \Cref{rem:tauqp} gives us  that
\[ \gsm(P_{m,w}(K)) \geq \tau(P_{m,w}(K))= \tau(Q^{m}(C_{w,1}(K)))=
\tau(C_{w,1}(K))+m= w \gsm(K)+m
\]
Since a single crossing change transforms $Q$ to a core of the solid torus, we have that $\gsm(Q(J)) \leq \gsm(J)+1$ for any knot $J$.
It is also easy to check that $\gsm(C_{w,1}(J)) \leq w \gsm(J)$ for any knot $J$, and so
\[ \gsm(P_{m,w}(K))= \gsm(Q^{m}(C_{w,1}(K))) \leq \gsm(C_{w,1}(K))+m\leq w \gsm(K)+m,
\]
and we have the desired equality.
\end{proof}
\begin{Example}\label{Ex:2cablesof2strandedtorusknots}
Let $p$ and $q$ be odd positive integers. We consider $C_{2,q}(T_{2,p})$, the $(2,q)$-cable of the $(2,p)$ torus knot. From another point of view, $C_{2,q}(T_{2,p})$ is the knot obtained as the closure of the
$4$-braid ${(a_2a_1a_3a_1)^p}{a_1}^{q-2p}$. Such a knot is  strongly quasipositive\footnote{For $q\geq0$, all (2,q)-cables of a non-trivial strongly quasipositive $K$ are strongly quasipositive since they are the boundary of a quasi-positive Seifert surface. Indeed, a Seifert surface is given as a $q$-fold positive Hopf plumbing on the zero framed annulus with core $K$. This Seifert surface is quasi-positive since positive Hopf plumbing preserves quasipositivity (see~\cite{Rudolph_98_QuasipositivePlumbing}) and the zero framed annulus with core $K$ is a quasi-positive Seifert surface (see~\cite[Lemma~1 and its proof]{rudolph_QPasObstruction}).} and as such has $g_3=\tau=\gsm$.
Concretely, \[\gsm(C_{2,q}(T_{2,p}))=g_3(C_{2,q}(T_{2,p}))=(q-1)/2+2g_3(T_{2,p})=\frac{q-1}{2}+p-1.\]

In contrast, we have as an application of \Cref{thm:mainIntro} that
\begin{equation}\label{eq:upperboundfor2cables}\g\left(C_{2,q}(T_{2,p})\right)\leq \gZ\left(C_{2,q}(T_{2,p})\right)\leq \gZ\left(C_{2,q}(U)\right)+\gZ(T_{2,p})=\frac{q-1}{2}+\frac{p-1}{2},\end{equation}

where the equality follows from $|\sigma/2|=\g=\gZ=g_3=(n-1)/2$ for $T_{2,n}$ with $n>1$ odd.
This a priori seems unexpected for all values of $q$. We discuss some special cases.

\textbf{Upper and lower bounds coincide on $\g$:} For $q=1$, \eqref{eq:upperboundfor2cables} is of course subsumed by~\Cref{cor:unknottedpattern}, and the inequalities are equalities. Similarly, upper and lower bounds agree for $p=1$, though this is less interesting since $C_{2,q}(T_{2,1})= T_{2,q}$. In fact, the lower bound for $\g(C_{2,q}(T_{2,p}))$ coming from Tristram-Levine signatures equals the upper bound of~\eqref{eq:upperboundfor2cables} when $q=1,3$ and any $p$, when $q=5$ and $p=3,5,7,9$, when $q=7$ and $p=3$, and for any $q$ when $p=1$.
Indeed, for $p,q\geq 3$ a Tristram-Levine signature calculation\footnote{Setting $\omega=e^{2\pi i t}$ with $t=\frac{1}{4}+\frac{1}{2p}-\epsilon$ and $t=\frac{q-2}{2q}+\epsilon$ for $\epsilon$ sufficiently small, we have \[\frac{\left|\sigma_\omega(C_{2,q}(T_{2,p}))\right|}{2}=\left\lfloor\frac{q}{4}+\frac{q}{2p}+\frac{1}{2}\right\rfloor+ \frac{p-1}{2}=\frac{q-1}{2}+\frac{p-1}{2}+\left\lfloor-\frac{q}{4}+\frac{q}{2p}+1\right\rfloor=\frac{q-1}{2}+\frac{p-1}{2}-\left\lfloor \frac{q}{4}-\frac{q}{2p}\right\rfloor\et\]
\[\frac{\left|\sigma_\omega(C_{2,q}(T_{2,p}))\right|}{2}
=\frac{q-1}{2}+\left\lfloor\frac{2p}{q}+\frac{1}{2}\right\rfloor
=\frac{p-1}{2}+\frac{q-1}{2}+\left\lfloor\frac{2p}{q}+\frac{1}{2}-\frac{p-1}{2}\right\rfloor
=\frac{p-1}{2}+\frac{q-1}{2}-\left\lfloor\frac{p}{2}-\frac{2p}{q}\right\rfloor,\text{ respectively.}\]
}
yields
\begin{equation}\label{eq:siglowerbounds}
\g(C_{2,q}(T_{2,p}))\geq \frac{q-1}{2}+\frac{p-1}{2}-\min\left\{\left\lfloor\frac{q}{4}-\frac{q}{2p}\right\rfloor,\left\lfloor\frac{p}{2}-\frac{2p}{q}\right\rfloor\right\},
\end{equation}
and one easily checks $\left\lfloor\frac{q}{4}-\frac{q}{2p}\right\rfloor=0$ if and only if $\left\lfloor\frac{p}{2}-\frac{2p}{q}\right\rfloor=0$ if and only if $\frac{1}{2}<\frac{1}{p}+\frac{2}{q}$. 

\textbf{Positive braid knots:} For $q= 2p+1$, $C_{2,2p+1}(T_{2,p})$ is the blackboard $+1$ framed cable of $T_{2,p}$ and as such the closure of a positive $4$-braid. (Indeed, ${(a_2a_1a_3a_1)^p}{a_1}^{q-2p}$ is evidently a positive $4$-braid for $q\geq 2p$.)
We find \begin{equation}\label{eq:pos4braid}
p+1 \leq \g(C_{2,2p+1}(T_{2,p}))\leq p+\frac{p-1}{2}=\frac{3p-1}{2}<\gsm\left(C_{2,2p+1}(T_{2,p})\right)=2p-1,\end{equation}
where the first inequality comes from~\Cref{eq:siglowerbounds} with $q=2p+1$.
This constitutes a significant difference between $\g$ and $\gsm$ for an infinite family of knots given as closures of a positive $4$-braid: for large $p$ the situation is
\[\frac{1}{2}\leq \lim_{p\to\infty}\frac{\g}{g_3}\left(C_{2,2p+1}(T_{2,p})\right)\overset{\text{\eqref{eq:pos4braid}}}{\leq}\frac{3}{4}< 1=\frac{\gsm}{g_3}.\]

We iterate the construction described above as follows.
For any positive braid $\beta$ of length $c$ with closure a knot $K$, one may consider the cable $C_{2,2c+1}(K)$.
This is the blackboard $+1$-framed cable of the standard diagram of $K$ coming from $\beta$ and hence is the closure of a positive braid of double the braid index of $\beta$ and length $4c+1$. We consider the result of $n$-times iterating this process starting with $K=T_{2,p}$ for $p \geq 3$ odd, by defining the knot
\[K_{n,p}=C_{2,2c_{n-1}+1}\left(C_{2,2c_{n-1}+1}\left(\cdots C_{2,2c_0+1}(T_{2,p})\cdots\right)\right),\]
where $c_0:=p$ and for $k \geq 1$ we define
\[c_k=4c_{k-1}+1=4(4c_{k-2}+1)+1= \cdots =4^kp+\frac{4^k-1}{3}.\]
Since $K_{n,p}$ is a positive knot, by applying Schubert's theorem for the 3-genus of a satellite knot we obtain
 \[\gsm(K_{n,p})=g_3(K_{n,p})= \sum_{k=0}^{n-1} c_k 2^{n-1-k} + \left(\frac{p-1}{2}\right)2^n=2^{2n-1}(p+ \frac{1}{3}) -2^n+ \frac{1}{3}.\]
Iteratively applying \Cref{prop:maininequality}, we find
\[\g(K_{n,p})\leq\gZ(K_{n,p})\leq\sum_{0}^{n-1}c_k+\frac{p-1}{2}=2^{2n-1}\left( \frac{2p}{3}+ \frac{2}{9}\right) - \frac{3n+1}{9} + \frac{p-3}{6}.\]
 Thus, we have
\[\lim_{n\to\infty}\frac{\g(K_{n,p})}{g_3(K_{n,p})}\leq \frac{2}{3} < 1=\frac{\gsm(K_{n,p})}{g_3(K_{n,p})}.\]

\textbf{Algebraic knots and torus knots:}
For $q= 4p+1$, $C_{2,4p+1}(T_{2,p})$ is an algebraic knot, which is smooth cobordism distance one from the torus knot $T_{4,2p+1}$. Consequently, \[\g(T_{4,2p+1})\leq  \g\left(C_{2,4p+1}(T_{2,p})\right)+1\overset{\text{\eqref{eq:upperboundfor2cables}}}{\leq} (5p-1)/2+1<3p=\gsm(T_{4,2p+1})=g_3(T_{4,2p+1}),\]
for $p>1$.
 This gives
\[\lim_{p\to\infty}\frac{\g(T_{4,2p+1})}{g_3(T_{4,2p+1})}\leq \frac{5}{6}.\] A priori, this is not particularly interesting since better upper bounds for $\g$ of torus knots with braid index $4$ were obtained in~\cite[Lemma~22(ii)]{BaaderFellerLewarkLiechti_15}. However, we find it noteworthy for two reasons.
Firstly, opposite the somewhat example based nature of the upper bounds from~\cite{BaaderFellerLewarkLiechti_15}, it is pleasant that no explicit Seifert matrix consideration for a specific knot is needed once Theorem~\ref{thm:mainIntro} is available. Secondly, by considering iterated cables of torus knots, one can find bounds on the topological 4-genera of torus knots of larger braid index that significantly improve the main results of~\cite{BaaderFellerLewarkLiechti_15}. However, this does not yield better results than those obtained by McCoy \cite{McCoy_19}, whose upper bounds on $\g(T_{p,q})$  for large $p,q$ improve any previous work; we refer the reader to his text for said bounds.
\end{Example}
\bibliographystyle{alpha}
\bibliography{peterbib,references}
\end{document}

%% file: Cable_4.pdf_tex
\begingroup%
  \makeatletter%
  \providecommand\color[2][]{%
    \errmessage{(Inkscape) Color is used for the text in Inkscape, but the package 'color.sty' is not loaded}%
    \renewcommand\color[2][]{}%
  }%
  \providecommand\transparent[1]{%
    \errmessage{(Inkscape) Transparency is used (non-zero) for the text in Inkscape, but the package 'transparent.sty' is not loaded}%
    \renewcommand\transparent[1]{}%
  }%
  \providecommand\rotatebox[2]{#2}%
  \ifx\svgwidth\undefined%
    \setlength{\unitlength}{2246.48842359bp}%
    \ifx\svgscale\undefined%
      \relax%
    \else%
      \setlength{\unitlength}{\unitlength * \real{\svgscale}}%
    \fi%
  \else%
    \setlength{\unitlength}{\svgwidth}%
  \fi%
  \global\let\svgwidth\undefined%
  \global\let\svgscale\undefined%
  \makeatother%
  \begin{picture}(1,0.70487273)%
    \put(0,0){\includegraphics[width=\unitlength,page=1]{Cable_4.pdf}}%
  \end{picture}%
\endgroup%

%% file: Trefoil.pdf_tex
\begingroup%
  \makeatletter%
  \providecommand\color[2][]{%
    \errmessage{(Inkscape) Color is used for the text in Inkscape, but the package 'color.sty' is not loaded}%
    \renewcommand\color[2][]{}%
  }%
  \providecommand\transparent[1]{%
    \errmessage{(Inkscape) Transparency is used (non-zero) for the text in Inkscape, but the package 'transparent.sty' is not loaded}%
    \renewcommand\transparent[1]{}%
  }%
  \providecommand\rotatebox[2]{#2}%
  \ifx\svgwidth\undefined%
    \setlength{\unitlength}{136.64002747bp}%
    \ifx\svgscale\undefined%
      \relax%
    \else%
      \setlength{\unitlength}{\unitlength * \real{\svgscale}}%
    \fi%
  \else%
    \setlength{\unitlength}{\svgwidth}%
  \fi%
  \global\let\svgwidth\undefined%
  \global\let\svgscale\undefined%
  \makeatother%
  \begin{picture}(1,1.04616503)%
    \put(0,0){\includegraphics[width=\unitlength,page=1]{Trefoil.pdf}}%
  \end{picture}%
\endgroup%

%% file: C_4Trefoil.pdf_tex
\begingroup%
  \makeatletter%
  \providecommand\color[2][]{%
    \errmessage{(Inkscape) Color is used for the text in Inkscape, but the package 'color.sty' is not loaded}%
    \renewcommand\color[2][]{}%
  }%
  \providecommand\transparent[1]{%
    \errmessage{(Inkscape) Transparency is used (non-zero) for the text in Inkscape, but the package 'transparent.sty' is not loaded}%
    \renewcommand\transparent[1]{}%
  }%
  \providecommand\rotatebox[2]{#2}%
  \ifx\svgwidth\undefined%
    \setlength{\unitlength}{597.90826953bp}%
    \ifx\svgscale\undefined%
      \relax%
    \else%
      \setlength{\unitlength}{\unitlength * \real{\svgscale}}%
    \fi%
  \else%
    \setlength{\unitlength}{\svgwidth}%
  \fi%
  \global\let\svgwidth\undefined%
  \global\let\svgscale\undefined%
  \makeatother%
  \begin{picture}(1,1.02939206)%
    \put(0,0){\includegraphics[width=\unitlength,page=1]{C_4Trefoil.pdf}}%
  \end{picture}%
\endgroup%

%% file: JandDisks.pdf_tex
\begingroup%
  \makeatletter%
  \providecommand\color[2][]{%
    \errmessage{(Inkscape) Color is used for the text in Inkscape, but the package 'color.sty' is not loaded}%
    \renewcommand\color[2][]{}%
  }%
  \providecommand\transparent[1]{%
    \errmessage{(Inkscape) Transparency is used (non-zero) for the text in Inkscape, but the package 'transparent.sty' is not loaded}%
    \renewcommand\transparent[1]{}%
  }%
  \providecommand\rotatebox[2]{#2}%
  \ifx\svgwidth\undefined%
    \setlength{\unitlength}{106.06700224bp}%
    \ifx\svgscale\undefined%
      \relax%
    \else%
      \setlength{\unitlength}{\unitlength * \real{\svgscale}}%
    \fi%
  \else%
    \setlength{\unitlength}{\svgwidth}%
  \fi%
  \global\let\svgwidth\undefined%
  \global\let\svgscale\undefined%
  \makeatother%
  \begin{picture}(1,2.08966359)%
    \put(0,0){\includegraphics[width=\unitlength,page=1]{JandDisks.pdf}}%
    \put(0.41041276,1.82638652){\color[rgb]{0,0,0}\makebox(0,0)[lb]{\smash{$-$}}}%
    \put(0.41041276,1.52799025){\color[rgb]{0,0,0}\makebox(0,0)[lb]{\smash{$+$}}}%
    \put(0.41041276,1.21545209){\color[rgb]{0,0,0}\makebox(0,0)[lb]{\smash{$-$}}}%
    \put(0.41041276,0.91705566){\color[rgb]{0,0,0}\makebox(0,0)[lb]{\smash{$-$}}}%
    \put(0.41041276,0.61865939){\color[rgb]{0,0,0}\makebox(0,0)[lb]{\smash{$+$}}}%
    \put(0.41041276,0.30612111){\color[rgb]{0,0,0}\makebox(0,0)[lb]{\smash{$+$}}}%
  \end{picture}%
\endgroup%

%% file: JandDisks_stab.pdf_tex
\begingroup%
  \makeatletter%
  \providecommand\color[2][]{%
    \errmessage{(Inkscape) Color is used for the text in Inkscape, but the package 'color.sty' is not loaded}%
    \renewcommand\color[2][]{}%
  }%
  \providecommand\transparent[1]{%
    \errmessage{(Inkscape) Transparency is used (non-zero) for the text in Inkscape, but the package 'transparent.sty' is not loaded}%
    \renewcommand\transparent[1]{}%
  }%
  \providecommand\rotatebox[2]{#2}%
  \ifx\svgwidth\undefined%
    \setlength{\unitlength}{106.27940458bp}%
    \ifx\svgscale\undefined%
      \relax%
    \else%
      \setlength{\unitlength}{\unitlength * \real{\svgscale}}%
    \fi%
  \else%
    \setlength{\unitlength}{\svgwidth}%
  \fi%
  \global\let\svgwidth\undefined%
  \global\let\svgscale\undefined%
  \makeatother%
  \begin{picture}(1,2.08549057)%
    \put(0,0){\includegraphics[width=\unitlength,page=1]{JandDisks_stab.pdf}}%
    \put(0.41159109,1.83846556){\color[rgb]{0,0,0}\makebox(0,0)[lb]{\smash{$-$}}}%
    \put(0.41159109,1.54066567){\color[rgb]{0,0,0}\makebox(0,0)[lb]{\smash{$+$}}}%
    \put(0.41159109,1.22875212){\color[rgb]{0,0,0}\makebox(0,0)[lb]{\smash{$-$}}}%
    \put(0.41159109,0.93095193){\color[rgb]{0,0,0}\makebox(0,0)[lb]{\smash{$-$}}}%
    \put(0.41159109,0.63315213){\color[rgb]{0,0,0}\makebox(0,0)[lb]{\smash{$+$}}}%
    \put(0.41159109,0.32123836){\color[rgb]{0,0,0}\makebox(0,0)[lb]{\smash{$+$}}}%
  \end{picture}%
\endgroup%